\documentclass[a4paper,11pt]{amsart}
\usepackage[top=2.2cm,left=2.8cm,right=2.8cm,bottom=2.8cm]{geometry}
\usepackage[foot]{amsaddr}

\usepackage{hyperref}
\hypersetup{
    colorlinks,
    citecolor=green,
    filecolor=black,
    linkcolor=blue,
    urlcolor=blue
}
\hypersetup{linktocpage}

\usepackage{cite}
\usepackage{graphicx}

\usepackage{amsbsy}
\usepackage{latexsym}
\usepackage{amsfonts}
\usepackage{amssymb}
\usepackage{upgreek}
\usepackage[usenames]{color}
\usepackage{amsmath,amsthm}
\usepackage{enumerate}

\usepackage{tikz}
\usetikzlibrary{arrows}
\usepackage{mathdots}
\usepackage{mathtools}

\usepackage{stmaryrd}
\usepackage{dsfont}

\usepackage{booktabs} 
\usepackage{pifont}
\usepackage{tikz}
\usetikzlibrary{arrows.meta,shapes.geometric,matrix,positioning,calc}

\DeclareMathOperator{\ran}{ran}

\DeclareMathOperator{\linspan}{span}

\DeclareMathOperator*{\Div}{div}

\providecommand{\abs}[1]{\lvert#1\rvert}

\providecommand{\norm}[1]{\lVert#1\rVert}
\providecommand{\bignorm}[1]{\bigl\lVert#1\bigr\rVert}
\providecommand{\Bignorm}[1]{\Bigl\lVert#1\Bigr\rVert}

\numberwithin{equation}{section}

\newtheorem{theorem}{Theorem}

\newtheorem{prop}[theorem]{Proposition}
\newtheorem{cor}[theorem]{Corollary}

\theoremstyle{definition}

\newtheorem{example}[theorem]{Example}

\theoremstyle{remark}

\theoremstyle{plain}

\newcommand{\R}{\mathbb{R}}

\newcommand{\N}{\mathbb{N}}

\newcommand{\Chi}{\raise .3ex
\hbox{\large $\chi$}}

\title[Numerical Stability in Neural Network Training]{Preconditioning and Numerical Stability in Neural Network Training for Parametric PDEs}
\author{}

\author{Markus Bachmayr$^1$}
\email{bachmayr@igpm.rwth-aachen.de}
\author{Wolfgang Dahmen$^{1,2}$}
\email{dahmen@igpm.rwth-aachen.de}
\author{Chenguang Duan$^{1}$}
\email{duan@igpm.rwth-aachen.de}
\author{Mathias Oster$^{1}$} 
\email{oster@igpm.rwth-aachen.de}

\address{$^1$ Institut f\"ur Geometrie und Praktische Mathematik, RWTH Aachen University, Im S\"usterfeld 2, 52072 Aachen, Germany}
\address{$^2$ University of South Carolina, Mathematics Department, 1523 Greene Street, Columbia, SC 29208, USA}

\thanks{The authors acknowledge funding by Deutsche Forschungsgemeinschaft (DFG, German Research Foundation) -- project number 442047500/SFB 1481 \emph{Sparsity and Singular Structures}. M.B.\ acknowledges funding by the European Union (ERC, COCOA, 101170147).
W.D. acknowledges funding by National Science Foundation grants FRG DMS-2245097,   DMS-2513112, and RTG DMS-2038080.}

\begin{document}

\maketitle

\vspace{-18pt}
\begin{abstract}
In the context of training  neural network-based approximations of solutions of parameter-dependent PDEs, we investigate the effect of preconditioning via well-conditioned frame representations of operators and demonstrate a significant improvement on the performance of standard training methods. We also observe that standard representations of preconditioned matrices are insufficient for obtaining numerical stability and propose a generally applicable form of stable representations that enables computations with single- and half-precision floating point numbers without loss of precision.


\end{abstract}

\section{Introduction}

\subsection{Problem formulation}
The need for computationally efficient approximations of solutions of parameter-dependent PDEs arises in many different applications. Abstractly, given a family of differential operators $\mathcal R_y$ in residual form,   parameterized by $y\in Y$,  
solutions $u(x,y)$  of $\mathcal R_y(u(\cdot, y) ) = 0$ can be viewed as functions of spatio-temporal variables $x\in \Omega \subset \R^d$ as well as of the parametric variables $y$
ranging over a given parameter domain $Y$. Here $Y$ can be a subset of a finite-dimensional  Euclidean space, $Y \subset \R^p$ with some $p\in\N$, or more generally of an infinite-dimensional sequence or function space.
In which sense   $u$ (or a derived quantity of interest) is to be approximated over $\Omega\times Y$, is determined by
   approximating the mapping $y \mapsto  u(y) = u(\cdot, y)$ in  appropriate norms
   for functions  on $Y$ and $\Omega$.  This task, sometimes referred to as ``operator learning,''   can arise both in a context of model reduction (as for the accelerated solution of inverse or optimization problems) and in uncertainty quantification.

Especially for challenging problems that are not amenable to linear model reduction methods such as POD or  reduced basis methods, approximations for $u$ based on neural networks are of interest. While in certain approaches such as physics-informed neural networks (PINNs), values of
$u$ are approximated by a neural network with input parameters $x$ and $y$,   it can be beneficial to use neural network representations only in the (potentially high-dimensional) $y$-variable, while the dependence on $x$ is approximated by well-understood classical discretization methods. Assuming $\{ \varphi_i \}_{i \in \mathcal I}$ with $\#\mathcal I < \infty$ to be a suitable set of fixed basis functions on $\Omega$, this leads to hybrid approximations of the form
\begin{equation}\label{eq:hybrid}
  u(x,y) \approx  \tilde u(x, y;\theta) = \sum_{i \in \mathcal I} \mathbf{u}_i(y; \theta) \varphi_i(x),
\end{equation}
where the approximation coefficients $\mathbf{u} = (\mathbf{u}_i)_{i \in \mathcal I}$ are jointly represented as neural networks and depend on $y$ as well as on a vector $\theta$ of trainable network parameters. Note that the basis functions $\varphi_i$ could be 
standard finite element basis functions or elements of a problem-adapted reduced basis. Representations of the form \eqref{eq:hybrid} have been used, for example, for problems with finite-dimensional $Y$ in \cite{GPRSK:21,KPRS:22,GHRK:25,BachmayrDahmenOster2024,CDG2025},
but are also common in the setting of \emph{operator learning} with infinite-dimensional $Y$ \cite{BHKS2021,KLLABSA2023}.

For any method that yields a quantification of the achieved accuracy of approximations with respect to a suitable norm, it is important to decide in which (weak) sense the equation
$\mathcal R_y(u(\cdot, y) ) = 0$ is required to hold. For a proper weak formulation, one should have $\mathcal R_y \colon U_y \to V_y'$ with Hilbert spaces $U_y$, $V_y$ for each $y$, postponig for a moment further comments on the choice of the trial- and test-spaces $U_y, V_y$. A commonly used   way to obtain the approximation $\tilde u$ to $y\mapsto u(y)\in U_y$ is then to employ {\em empirical risk minimization}
\begin{equation}\label{eq:solutionregression}
    \min_\theta \sum_{k=1}^K \norm{ u(y_k)  - \tilde u(y_k; \theta) }_{U_{y_k}}^2,
\end{equation}
see, e.g.  \cite{GPRSK:21,KPRS:22}, where for the sake of computational convenience,
the norms $\|\cdot\|_{U_y}$ are sometimes replaced by $L_2$-norms.
Viewing the parameters as random variables with probability distribution $\mu$ and assuming
the $y_k$, $k = 1,\ldots, K$, to be i.i.d.\ draws from $\mu$, this corresponds to approximately minimizing a Monte Carlo approximation of
the expectation $\mathbb{E}_{y\sim \mu}\big[\|u(y)- \tilde u(y_k; \theta)\|^2_{U_{y_k}}\big]$.
If the norms $\|\cdot\|_{U_y}$ are uniformly equivalent over $Y$, the expectation is a standard Bochner norm. If not, under mild conditions on the scalar products, inducing $\|\cdot\|_{U_y}$, the above expectation is still well-defined as a so-called direct integral, see  \cite{BachmayrDahmenOster2024}. 

Alternatively, to mitigate the computational cost of a large number of high-fidelity training
samples $u(y_k)$, 
 it can also be advantageous to instead use {\em residual losses}. To infer from a residual the accuracy of the approximation it is essential in which norm the residual
is measured. As shown in \cite{BachmayrDahmenOster2024}, under the above assumptions on the mapping properties of $\mathcal R_y$, a proper choice is $\norm{\mathcal R_y (w)}_{V_y'}$ (the dual test norm), because if the underlying weak formulations are stable one has  $\norm{\mathcal R_y (w)}_{V_y'} \eqsim \norm{ w }_{U_y}$, a property that we refer to as \emph{variational correctness}. Hence the size of the loss is  a reliable and efficient a posteriori error
bound. Then \eqref{eq:solutionregression} can be replaced by
\begin{equation}\label{eq:residualregression}
    \min_\theta \sum_{k=1}^K \bignorm{ \mathcal R_{y_k}\bigl( \tilde u(y_k; \theta) \bigr) }_{V_{y_k}'}^2.
\end{equation}

The issue that we address in this paper concerns both types of optimization problems
\eqref{eq:solutionregression} and \eqref{eq:residualregression}. The quality of corresponding numerical results depends crucially on the performance of optimization schemes used for the loss minimization. While the nonlinear structure of neural networks poses an unavoidable obstacle, the separation of spatial and parametric variables, facilitated by the hybrid format \eqref{eq:hybrid}, nevertheless offers substantial
advantages. A central point in what follows
is that each summand in the objective is a quadratic form that offers room
 for formulating the objective such that it becomes favorable for descent strategies. An intrinsic obstruction lies in the fact that the norms on $U_{y}$ and $V_y$ are typically induced by  differential operators, for example by the Laplacian when $U_{y} = H^1_0(\Omega)$. We explain  later below in more detail that
 both types of loss \eqref{eq:solutionregression} and \eqref{eq:residualregression}
exhibit similar features with regard to the dependence on differential operators,
opening the door to exploiting {\em preconditioning} techniques for removing PDE-related ill-conditioning from the optimization problems.

\subsection{Choice of basis, preconditioning, and performance of optimization schemes}
\newcommand{\cI}{\mathcal{I}}
The most basic and well-understood scenario is $\mathcal R_y v = A_y v - f$, where $A_y\colon U_y \to V_y'$ is a linear operator and $f \in V_y'$. This corresponds to a weak formulation
\[
  \langle A_y u(y) , v\rangle = f(v), \quad \forall \,v \in V_y, \; y \in Y.
\] 
Variational correctness in the above sense is then equivalent to the well-posedness
of this weak formulation (that is, the conditions in the Babu\v{s}ka-Ne\v{c}as theorem hold), which 
in turn means that $A_y : U_y\to V'_y$ is a norm isomorphism. 

For simplicity of exposition, we assume that the spaces $U_y$, $V_y$ agree for all $y$ with equivalent norms, so that we can
suppress dependence on $y$ in corresponding notations. Note that then the problems
\begin{equation}
\label{U}
\min_{w\in U}\|A_y w- f\|^2_{V'}\quad \mbox{or}\quad \min_{w\in U}\|u(y)-w\|^2_U
\end{equation}
are {\em well-conditioned} when posed in $V'$, $U$, respectively, so that descent methods in $U$ or $V'$ would converge rapidly. However, as soon as these problems are posed over some finite dimensional
trial space 
\[
U_\cI = {\rm span}\,\Phi,\quad  \Phi = \{ \varphi_i \colon i  \in \mathcal I\}, 
\]
the objective becomes a quadratic form over $\R^{\#\cI}$ whose numerical conditioning now depends
on the choice of the representation system $\Phi$ in the same way as it affects 
the performance of iterative solvers for resulting linear systems, namely through the
condition number of the Gramian $(\langle \varphi_i, \varphi_{i'}\rangle_U )_{i,i' \in \mathcal I}$, associated to $\Phi$. Obviously, an ideal choice of $\Phi$ would be a subset of a
$U$-{\em orthonormal} system that would exactly preserve the properties of the problem on \eqref{U} and lead to resulting condition numbers uniformly bounded in $\#\mathcal I$.
While suitable reduced basis systems may have  this property, this is not true 
for a generic finite element basis, which is the primary concern in what follows.

 We remark that if for $V'\neq V$, the residual loss $\|A_y w- f\|^2_{V'}$ is replaced by the simpler one
$\|A_y w- f\|^2_{L_2}$ as in PINN, $A_y$ is not necessarily an isomorphism onto $L_2$, so that even the continuous problem on $U$ is ill-posed, adding
to the difficulties incurred by subsequent discretizations.

{\em Preconditioning} the Gramian $(\langle \varphi_i, \varphi_{i'}\rangle_U )_{i,i' \in \mathcal I}$ is of obvious importance for reducing the number of descent steps in
a minimization and thus a first subject of this paper. However, there is a second issue
that turns out to play a decisive role, especially in connection with optimization, or more precisely,
when dealing with nonlinear parameterizations for the solution coefficients $\mathbf u$, namely the possible {\em loss of numerical precision}, due to the standard application of a preconditioner.

It needs to be noted  that this is indeed a separate issue from obtaining a well-conditioned problem for the coefficients $\mathbf{u}$; see, for example, related discussions in 
\cite{BK20,MR4839137}.  We will demonstrate how 
an appropriate combination of a variational formulation for the family of PDEs and a representation systems $\Phi$ for the spatial discretizations affects several points in the solution process: the {\em norms} with respect to which errors can be controlled,   effective {\em initializations} and the {\em converence} of optimization schemes, as well as  limitations on attainable precision in floating point arithmetic.

While difficulties in reducing the objective  can to some degree be mitigated by more sophisticated optimization methods, for example the Gauss-Newton (in this context also known as \emph{natural gradient descent}),   such techniques do not address the loss of significant digits in the result due to the action of ill-conditioned matrices. Especially when using lower-precision arithmetic for using neural network training, this will be seen to  severely limit the attainable accuracies.

\subsection{Stable representations and relation to prior works}

As we demonstrate theoretically and quantify numerically in the following sections, well-posed variational formulations combined with well-conditioned frame representations can lead to greatly improved convergence of optimization methods. 
However, we stress that preconditioning by itself does not automatically lead to satisfactory results when using lower machine precision, such as 32-bit or even 16-bit floating point numbers. As an additional ingredient, recovering full-precision results even for low machine precision,  requires {\em well-conditioned decompositions} of discretization matrices, which we refer to as {\em stable representations} and explain below in  more detail. 

The approach proposed in this paper can be seen in the context of   several prior works on similar issues, albeit from
different perspectives and for varying applications. In the context of iteratively solving
large linear systems, it is well understood that   a good preconditioner may significantly speed up
the convergence of iterative schemes, but may fail to achieve high accuracy in the numerical solution. The reason is that   the preconditioned matrix may be given as a product of several matrices which themselves are still ill-conditioned; see, for example, the references in \cite{MR4839137}.
For systems arising from discretizing an operator equation, this issue has been addressed systematically by so-called {\em full operator preconditioning} (FOP) in \cite{MR4839137}. The key idea is to first reformulate a given operator equation on the continuous level by means of multiplying the original operator from the left and the right by suitable operators in such a way that the resulting product becomes well-conditioned on spaces for which suitable trial and test bases are available (such as $L_2$-spaces).

The prospective gain lies in avoiding separate discretizations of ill-conditioned factors.
In summary, to warrant both rapid convergence of the preconditioned iteration as well as high numerical accuracy, the condition of the transformed operator equation as well as the condition of
trial basis should be moderate or small. The present paper shares the central objective of avoiding adverse effects of ill-conditioned factors in a preconditioned matrix. 
However, an important distinction is that for the type
of problems considered here, we do not need to first reformulate an operator equation
on the continuous level. 
Instead, it suffices to use a non-standard factorization of standard preconditioned matrix representation. To the best of our knowledge, this fact has been observed first in the 
context  of multilevel tensor representations   in \cite{BK20}, which is the main source of inspiration also for the present paper. 
Moreover, the role of frames in the present context can be related to ``suitability'' of trial bases in the FOP context and hence is a 
task to be addressed in either approach.  

The main mechanism used here has subsequently also been exploited in the quite different context 
of quantum algorithms for finite element systems \cite{DP25}. There, a central role is played by so-called block encodings of matrices revolving around projections on unitary representations under normalization constraints. The complexity of such algorithms, determined by the number of iterations,   is  governed by
the effective condition of such block encodings, which again depends on factorizations with
well-conditioned factors. Here, ill-conditioned factors cause a loss in precision due to projection and normalization constraints.

From the point of view of only speeding up convergence, there are numerous other
related works, most notably in the context of machine learning and neural operator learning.
For instance, for neural network approximations of parameter-dependent problems, preconditioning also enters the discussion in \cite{GHRK:25}. In this case, left preconditioning is combined with a least squares formulation to obtain a well-conditioned problem. However, this amounts to using an $L_2$-space as trial space, so that errors can only be controlled with respect to this norm.
Under the flag of natural gradient flows in the context of supervised learning, the approximate inversion of the Gramian of the weight gradient $\nabla_\theta u(\cdot;\theta)$
can be viewed as preconditioning; see, for example, \cite{Cayci,Wuchen,Peher1,DLTW2025}. The scalable solution of the arising subproblems, however, again hinges on suitable preconditioning for these Gramians.

   In the present paper, we show   that also standard multilevel finite element preconditioners can be brought into alternative matrix forms -- computable with the same ease and efficiency as their standard counterparts -- that lead to numerically stable methods without loss of accuracy even  for low machine precision. 
   Hence, approximating the spectrum of Hermitian squares of operators as in 
    \cite{MR4839137} %
    is  not required for avoiding a loss of precision. As shown below, it can instead suffice to re-express standard preconditioned matrices in terms of slightly different alternative representations.

\section{Norm stability of representation systems and condition numbers}

We first discuss the influence of the representation system $\Phi = \{ \varphi_i \colon i  \in \mathcal I\}$ in
\eqref{eq:hybrid} on the performance of optimization methods for \eqref{eq:solutionregression} and \eqref{eq:residualregression}.
For simplicity of exposition,   we focus on Galerkin discretizations of elliptic problems. 
In these problems, the identical trial and test spaces can be chosen to be independent of the parameter and are denoted by $U$.
The basic principles apply more generally to Petrov-Galerkin-type methods, where trial and test spaces differ and may depend (even as sets) on $y$.

In order to explain  the basic concepts, we use the classical model problem
\begin{equation}
  - \Div \bigl(a_y \nabla u(y) \bigr) = f\quad \text{ in $\Omega$,} \qquad u|_{\partial\Omega} = 0,
\end{equation}
in appropriate weak formulation, where the diffusion field $a_y$ is bounded and strictly positive uniformly with respect to $y \in Y$.

\subsection{Examplary variational formulations}

In what follows, we make use of a common structure in weak formulations. We assume that the problem can be written in the form: find $u \colon Y \mapsto U$ such that
\begin{equation}\label{eq:generalform}
    \langle \mathcal D u(y), \mathcal D v \rangle_y = \langle \ell, v \rangle \quad \text{for all $v \in U$, $y \in Y$.} 
\end{equation}
Here we require the $y$-independent mapping $\mathcal D \colon U \to H$, with $H = \bigl(L_2(\Omega)\bigr)^m$ for appropriate $m \in \N$,  to be bounded and injective such that $\mathcal D \colon U \to \ran(\mathcal D) \subset H$ is an isomorphism. Moreover,  we assume $\langle \cdot, \cdot \rangle_y$ to be a bilinear form on $H$ that may depend on $y$. We are interested
in equivalently expressing \eqref{eq:generalform} as a quadratic minimization problem.

\begin{example}
\label{ex:E}
We assume that for $a_y\in L_\infty(\Omega)$ there exist constants $a_1, a_\infty$ such that $0<a_1\le a_y(x)\le a_\infty <\infty$ for all $x\in\Omega$ and $y \in Y$.
Then for $U= H^1_0(\Omega)$, $\mathcal D = \nabla$, $m = d$ and $\langle \cdot, \cdot \rangle_y = \langle a_y \cdot, \cdot\rangle_{(L_2)^m}$, the above assumptions are valid and the solution to 
\eqref{eq:generalform} is characterized as the unique minimizer of
the standard energy minimization
\begin{equation}\tag{E}\label{eq:E}
   \min_{u \in H^1_0(\Omega) } \bigl\{  \textstyle\frac12 \displaystyle \langle a_y \nabla u, \nabla u \rangle_{L_2} - \langle f, u\rangle \bigr\}.
\end{equation}
Moreover, we have
\begin{equation}
\label{stabvar}
c_U\|u\|^2_{U}\leq \langle \mathcal Du,\mathcal Du\rangle_y	\leq C_U\|u\|^2_{U}\quad  \forall \,u\in U, \,y\in Y
\end{equation}
with constants $0<c_U \le C_U<\infty$ depending only on $\Omega$, $a_1$, and $a_\infty$.
\end{example}

\begin{example}
\label{ex:LS}
Under the assumptions of Example \ref{ex:E}, introducing $\sigma := a_y\nabla u$ as an auxiliary variable, the solution to \eqref{eq:E}
can also be obtained via the
{\em first-order system least squares} (FOSLS) formulation 
\begin{equation}\tag{LS}\label{eq:LS}
  \min_{ \substack{ u \in H^1_0(\Omega)  \\ \sigma \in H(\Div,\Omega) }} \Bigl\{  \norm{ \Div \sigma - f }_{L_2}^2 + \norm{a_y^{-1} \sigma - \nabla u }_{L_2}^2 \Bigr\}.
\end{equation}
In this case, $U = H^1_0(\Omega) \times H(\Div,\Omega)$,  $m=2d+1$, and \eqref{stabvar} holds for 
\[
  \mathcal D (v, \tau) = \begin{pmatrix}  \nabla v  \\ \tau \\ \Div \tau \end{pmatrix}, \quad 
    \bigl\langle (g, \tau, p) , (h, \rho, q)   \bigr\rangle_y  = \langle  p, q  \rangle_{L^2} + \langle a_y^{-1} \tau - g,  a_y^{-1} \rho - h \rangle_{(L^2)^d}.
\]
In fact,
 by \cite[Thm.~5.15]{MR2490235} (which was originally obtained in \cite{FOSL}) there exist $c_U,C_U>0$ depending on $a_1$ and $a_\infty$ in Example \ref{ex:E} such that 
\begin{align*}
	\bigl\langle (\nabla u, \tau,\mathrm{div}(\tau)),(\nabla u, \tau,\mathrm{div}(\tau))\bigr\rangle_y  \geq  c_U \bigl(\|\tau\|_{H(\mathrm{div})}^{2}+\|u\|_{H^1_0(\Omega)}^{2} \bigr).
\end{align*}
Similarly, by the upper bound on $a_y$ and the Cauchy-Schwarz inequality, we have 
\begin{align*}
	\bigl\langle (\nabla u,\tau,\mathrm{div}(\tau)),(\nabla v,\sigma,\mathrm{div}(\sigma))\bigr \rangle_y \leq C_{U}\bigl( \|\mathrm{div}(\tau)\|\|\mathrm{div}(\sigma)\|+(\|\tau\|+\|\nabla u\|)(\|\sigma\|+\|\nabla v\|) \bigr),
\end{align*}
see  \cite{MR2490235,FOSL,BachmayrDahmenOster2024} for further details.
\end{example}

One advantage of \eqref{eq:LS} is that information on the flux $\sigma$, which is often of primary interest, is obtained without post-processing (which may diminish accuracy) of the solution $u$ to \eqref{eq:E}. In the context of variationally correct residual regression \cite{BachmayrDahmenOster2024}, the following observation is essential:
while substituting an approximate solution into the energy functional \eqref{eq:E}
does not show its accuracy, the formulation \eqref{eq:LS} has the advantage of directly giving a quantification of the total solution error. Indeed, for any $v \in H^1_0(\Omega)$, $\tau \in H(\operatorname{div},\Omega)$, for the error with respect to the solution $(u,\sigma)$ of \eqref{eq:LS}, we have
\begin{equation}\label{eq:LSreserr}
   \norm{ u - v }_{H^1_0}^2 + \norm{ \sigma - \tau }_{ H(\operatorname{div}) }^2  
     \eqsim \norm{ \Div \sigma - f }_{L_2}^2 + \norm{a_y^{-1} \sigma - \nabla u }_{L_2}^2 ,
\end{equation}
with proportionality constrants depending only on $a_1,a_\infty$, and $\Omega$.

Note that $\ell$ in \eqref{eq:generalform} may generally also depend on the parameter $y$.
We dispense here with including this case since it is not essential for
subsequent considerations. Least squares formulations similar to \eqref{eq:LS} exist also for parabolic problems \cite{FK2021} and for acoustic wave equations \cite{FGK2023}, and we expect that the results given here can be transferred directly to FOSLS formulations of such problems.

\subsection{Riesz basis and frame representations}

With a choice of $\Phi$, the problems \eqref{eq:E} and \eqref{eq:LS} can both be written in the form
\begin{equation}\label{eq:residualregression}
    \min_\theta \sum_{k=1}^K  \bigl( \langle \mathbf A_{y_k} \mathbf u(y_k; \theta),  \mathbf u(y_k; \theta) \rangle_{\ell_2}  - \langle \mathbf f ,  \mathbf u(y_k; \theta) \rangle_{\ell_2} \bigr) ,
\end{equation}
based on \eqref{eq:generalform}, where
\begin{equation}
 \mathbf A_y = \bigl(  \langle \mathcal D \varphi_{i'}, \mathcal D \varphi_i \rangle_{y} \bigr)_{i, i' \in \mathcal I}, \quad \mathbf f = \bigl( \langle \ell, \varphi_i \rangle \bigr)_{i \in \mathcal I}  \,.
\end{equation}
Using regression for the solution as in \eqref{eq:solutionregression} leads to a problem of the same structure, except that $\mathbf A_y$ does not depend on $y$ because the inner product providing the entries of the Gramian can be chosen to be independent of $y$.

It is well known that standard finite element basis functions lead to ill-conditioned discretization matrices $\mathbf A_y$, since they form well-conditioned bases in $L_2$ but generally not in $U$.
This problem can be circumvented by using Riesz bases $\Phi$ (see, for example, \cite{Dacta,DK}), that is, $\varphi_i$ for $i \in \mathcal I$, are required to satisfy
\begin{equation}
   c_\Phi   \norm{\mathbf v}_{\ell_2}  \leq  \Bignorm{ \sum_{i \in \mathcal I} \mathbf v_i \varphi_i }_U  \leq C_\Phi   \norm{\mathbf v}_{\ell_2} .
\end{equation}
Examples of Riesz bases for Sobolev spaces $U$ are wavelet or Fourier bases, as constructed in
\cite{DSchn,DS}  on polygonal domains.
However,  the construction of such bases   becomes   technically more demanding
with increasing geometric complexity of the domain.

The Riesz basis property means that the {\em frame operator}
\begin{equation}\label{eq:F}
  F \colon \ell_2(\mathcal I) \to U, \quad  \mathbf v \mapsto \sum_{i \in \mathcal I} \mathbf v_i \varphi_i
\end{equation}
is an isomorphism from $\ell_2(\mathcal I)$ to $U_\Phi = \overline{\linspan \Phi}$ and 
$\|F\|_{\ell_2\to U}\le C_\Phi$.
This condition can be relaxed by instead allowing $\Phi$ to be {\em overcomplete}, but such that $F$ is still an isomorphism from $\ker(F)^\bot$ to $U_\Phi$, or equivalently, that $F' \colon U_\Phi' \to \ran(F') \subset \ell_2(\mathcal I)$ is an isomorphism. As one can show, this in turn is equivalent to the \emph{frame property}
\begin{equation}\label{eq:frameprop}
 c_\Phi^2  \norm{ h }^2_{U'_\Phi}  \leq \sum_{i \in \mathcal I} \abs{ h(\varphi_i) }^2  \leq  C^2_\Phi  \norm{ h }_{U'_\Phi}^2 \quad\text{for all $h \in U_\Phi'$,}
\end{equation}
which  we adopt here  as a definition of a frame $\Phi$  for $U_\Phi \subset U$. 

We refer to 
$c_\Phi, C_\Phi$ as corresponding frame constants. Such frames are much easier to construct and play a pivotal role in the theory of {\em additive Schwarz methods} addressed below; for an overview, we refer to \cite{Oswald}.

\subsection{Multilevel frames}\label{sec:multilevel}
One way to construct stable frames is based on multilevel structures in finite element methods. Let $W_j \subset H^1(\Omega)$ for $j \in \N$ be a hierarchy of finite element subspaces such that $W_j \subset W_{j+1}$ with standard $H^{1}$-normalized Lagrange basis $\{ \varphi_{j, k} \colon k \in \mathcal I_j \}$ for each $W_j$.
The following is a classical result in the context of additive Schwarz methods; see, for example, \cite{Oswald,BPX,HSS:08}.

\begin{theorem}\label{thm:bpx}
 Let $\norm{ \varphi_{j,k} }_{H^1(\Omega)} \eqsim 1$. Then $\{ \varphi_{j,k} \colon j = 1, \ldots, J , \; k \in \mathcal I_j \}$ is a frame for $W_J \subset H^1(\Omega)$ with frame constants independent of $J$.
\end{theorem}

The stiffness matrix with respect to the finest discretization level $J$ reads
\begin{equation}
 \mathbf A_y = \bigl( \langle \mathcal D \varphi_{J,k},  \mathcal D \varphi_{J, k'} \rangle_{y} \bigr)_{k,k' \in \mathcal I_J}. 
\end{equation}
Let  $N = \# \mathcal I_J$ and $\hat N = \sum_{j=1}^J \# \mathcal I_j$. We  denote
by $\mathbf H$ the mapping
 that takes the frame expansion coefficients into the coefficients with respect to the basis of $W_J$, that is, 
\begin{equation}
   \mathbf H \colon \R^{\hat N} \to \R^{N}, \quad \mathbf w = (\mathbf w_{j, k} )_{\substack{j=1,\ldots,J\\ k \in \mathcal I_j}}  \mapsto \mathbf v = (\mathbf v_{k})_{k \in \mathcal I_J}
\end{equation}
is defined by the property that
\[
    \sum_{j=1}^J \sum_{k \in \mathcal I_j} \mathbf w_{j,k} \varphi_{j,k} = \sum_{k \in \mathcal I_J} \mathbf v_k \varphi_{J,k}.
\]
As an immediate consequence of Theorem \ref{thm:bpx}, we obtain
\begin{equation}
\label{AS}
   \operatorname{cond}_2 (\mathbf H^\intercal \mathbf A_y \mathbf H ) \lesssim 1, 
\end{equation}
where $\mathbf H$ is often referred to as Bramble-Pasciak-Xu (BPX) preconditioner \cite{MR1023042,BPX}.

Under mild conditions on the diffusion coefficients, these  concepts lead to solvers for discretized elliptic problems within discretization error accuracy that require {\em linear time}, that is, the computational work is proportional to the system size. More precisely, this holds when discretization errors dominate the effect of floating point arithmetic (see, for example, the discussion in 
\cite{MR4839137}). 

\subsection{Effect of preconditioning on convergence of optimization schemes}
The observations made above on solvers for linear systems in essence carry over to the present context of
optimization.
In fact, using a well-conditioned representation also improves the convergence behaviour of optimization schemes for neural network approximations, as demonstrated in Table~\ref{tab:with:without:preconditioning} (for the definitions of the error measures used, see Section \ref{sec:exp}). However, the same reservations regarding floating point precision as in the solver context apply, which is the issue addressed in the remainder
of the paper. 

\begin{table}[htbp]
\centering
\caption{Comparison of FOSLS with and without preconditioning.}
\label{tab:with:without:preconditioning}
\small
\setlength{\tabcolsep}{4pt}
\begin{tabular}{lccccc}
\toprule
Optimizer & Precond. & Precision & MRE & MSE & Loss \\
\midrule 
SGD & \ding{56} & \texttt{float32} & $7.35\times 10^{-1}$ & $8.53\times 10^{1}$ & $7.50\times 10^{-1}$ \\
 & \ding{56} & \texttt{float64} & $9.35\times 10^{-1}$ & $8.53\times 10^{1}$ & $7.51\times 10^{-1}$ \\
 & \ding{51} & \texttt{float32} & $1.45\times 10^{-1}$ & $2.62$ & $1.06\times 10^{-2}$ \\
 & \ding{51} & \texttt{float64} & $1.44\times 10^{-1}$ & $2.60$ & $1.06\times 10^{-2}$ \\
\midrule
Adam & \ding{56} & \texttt{float32} & $1.59\times 10^{-1}$ & $3.17$ & $1.78\times 10^{-2}$ \\
 & \ding{56} & \texttt{float64} & $1.57\times 10^{-1}$ & $3.09$ & $1.67\times 10^{-2}$ \\
 & \ding{51} & \texttt{float32} & $1.01\times 10^{-3}$ & $1.52\times 10^{-4}$ & $1.19\times 10^{-6}$ \\
 & \ding{51} & \texttt{float64} & $1.26\times 10^{-3}$ & $2.20\times 10^{-4}$ & $1.14\times 10^{-6}$ \\
\midrule
L-BFGS & \ding{56} & \texttt{float32} & $1.37\times 10^{-1}$ & $2.33$ & $1.04\times 10^{-2}$ \\
 & \ding{56} & \texttt{float64} & $1.37\times 10^{-1}$ & $2.32$ & $9.34\times 10^{-3}$ \\
 & \ding{51} & \texttt{float32} & $4.76\times 10^{-3}$ & $3.32\times 10^{-3}$ & $3.43\times 10^{-5}$ \\
 & \ding{51} & \texttt{float64} & $1.33\times 10^{-3}$ & $2.99\times 10^{-4}$ & $2.93\times 10^{-6}$ \\
\midrule
NGD & \ding{56} & \texttt{float32} & $9.65\times 10^{-1}$ & $9.08\times 10^{1}$ & $8.05\times 10^{-1}$ \\
 & \ding{56} & \texttt{float64} & $9.65\times 10^{-1}$ & $9.08\times 10^{1}$ & $8.06\times 10^{-1}$ \\
 & \ding{51} & \texttt{float32} & $3.89\times 10^{-3}$ & $2.38\times 10^{-3}$ & $1.20\times 10^{-5}$ \\
 & \ding{51} & \texttt{float64} & $2.78\times 10^{-3}$ & $1.26\times 10^{-3}$ & $6.08\times 10^{-6}$ \\
\bottomrule
\end{tabular}
\end{table}

\subsection{Effect on random initializations}

In the optimization over a set of neural network parameters $\theta$ in \eqref{eq:residualregression}, also the random initialization of $\theta$ plays an important role. The network weights in $\theta$ are often chosen to be appropriately scaled i.i.d.\ samples of scalar Gaussian distributions.
As illustrated in Figure~\ref{fig:initialization}, using such initializations to parameterize the coefficients of the frame representations in $H^1$ leads to starting values representing substantially more regular functions than with standard finite element basis functions. This means that when using the preconditioned form of the problem, the optimization starts from significantly lower loss values. 

This observation can be interpreted in the context of representations of random fields: whereas linear combinations of standard finite element basis functions with i.i.d.\ standard Gaussian coefficients lead to approximations of white noise, combining such coefficients with frames in $H^1$ yields random fields having the regularity properties of Brownian motion. Although the random parameterization in terms of neural networks shown in Figure \ref{fig:initialization} (here for the full ResNet architecture as described in Section \ref{section:stable:structure} with the initialization described above) is substantially more involved, we observe similar regularity of sample paths.

\begin{figure}[htbp]
\centering
\includegraphics[width=1.0\linewidth]{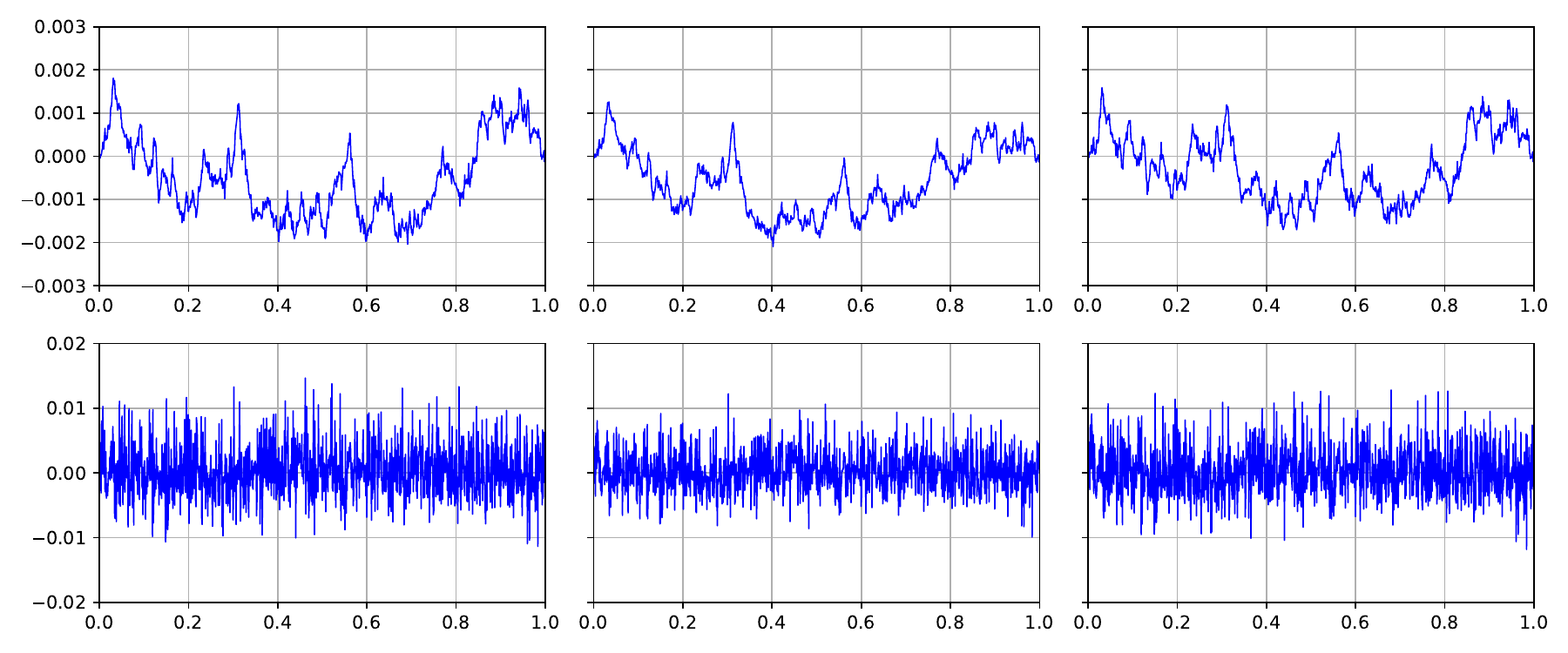}
\caption{Random initialization with and without frame representation. (\textbf{Top}) Initialization with frame representation. (\textbf{Bottom}) Initialization without frame representation.}
\label{fig:initialization}
\end{figure}

\section{Numerical stability of preconditioned problems}

Although the matrix $\mathbf H^\intercal \mathbf A_y \mathbf H$ considered above is well-conditioned, preconditioning in this form is in general not sufficient for numerical stability of methods, see also the discussion in \cite{MR4839137}. While the convergence of iterative methods depending on the condition number may be improved, the achievable accuracy remains limited because when applied in this form, the action of the ill-conditioned matrices $\mathbf A_y$ and $\mathbf H$ still leads to a loss of precision.
In this regard, standard recursive implementations of multilevel preconditioners (with linear costs) correspond to decomposing the well-conditioned symmetrically preconditioned matrix into the product $\mathbf H^\intercal \mathbf A_y \mathbf H$ of individually ill-conditioned matrices.
However, we now consider how such ill-conditioned matrices can be completely avoided not only by operator preconditioning as proposed in \cite{MR4839137}, but also in the context of standard multilevel finite element bases.

In what follows, we assume a family of frames $\Phi_J$ in the trial space $U$ with refinement parameter $J \in \N$ to be given such that for the corresponding constants $c_{\Phi_J} , C_{\Phi_J} >0$ in \eqref{eq:frameprop}, 
\begin{equation}\label{eq:unif}
  \sup_{J \in \N} \frac{C_{\Phi_J}}{c_{\Phi_J}} < \infty.
\end{equation}
As an example, we subsequently consider multilevel frames as in Section \ref{sec:multilevel}, where $J$ corresponds to the maximum discretization level.

\subsection{Stable representations}

In what follows, suppressing the index $J$ and thus writing $\Phi = \Phi_J$, for 
$F$ as in \eqref{eq:F}, we write
 \[ 
  H_\Phi = \ran \mathcal{D}F. 
  \]
Note that since $\mathcal D$ is an operator of vector-valued weak derivatives, when $U_\Phi = \linspan \Phi \subset U$ is a (finite-dimensional) subspace of piecewise polynomials, then also $H_\Phi$ is a subspace of piecewise polynomials. 

\begin{prop}
\label{prop:stab}
Let $\Phi$ satisfy \eqref{eq:unif} and assume that the variational formulation \eqref{eq:generalform} is uniformly stable for $y\in Y$, that is,
 \[
 c_U\|u\|^2_{U}\leq \langle \mathcal Du,\mathcal Du\rangle_y	\leq C_U\|u\|^2_{U}\quad  \forall \,u\in U, \,y\in Y.
 \]
 Then with constants independent of $J$, the following hold true:
\begin{enumerate}[\rm(i)]
\item $\mathcal{D} F  \colon \ell_2(\mathcal{I}) \to  H$
 is uniformly well-conditioned as a mapping from $\ker ( F )^\bot$ to $H_\Phi$.
 \item Moreover, in both \eqref{eq:E} and \eqref{eq:LS}, the bilinear form $(h, \tilde h) \mapsto \langle h, \tilde h\rangle_y$ is bounded and elliptic on $H_\Phi$ uniformly in $y \in Y$.
\end{enumerate}
 \end{prop}

\begin{proof}
Regarding part (i), 
we have that for all $u,v \in \ker{F}^\perp$ we can bound the spectrum of $\mathcal DF$ as follows
\[
\sup_{v\in \ker( F)^\perp}\frac{\langle \mathcal DFv,\mathcal DFv\rangle_y}{\|v\|_U^2}\leq  C_{U}\frac{\| Fv\|^2_{U}}{\|v\|^2_{U}%
} \leq C_{U} \|F\|^2_{\ell_2\to U}\leq C_{U}C^2_{\Phi_J}.
 \]
Denoting by $P^\perp$ the $U$-projection onto $\ker(F)^\perp$, we have   
\[
\inf_{v\in U,\|v\|_U=1}\frac{\langle \mathcal DFv,\mathcal DFv\rangle_y}{\|v\|_U} \geq c_U\frac{ \|F P^\perp v\|^2_{U}}{\|P^\perp v\|^2_{U}} \geq c_{U} c^2_{\Phi_J}.
\]
Hence, we have $\mathrm{cond}((\mathcal D F)'\mathcal DF)\leq \frac{C_UC^2_{\Phi_J}}{c_Uc^2_{\Phi_J}}$.

Concerning (ii), in both cases \eqref{eq:E} and \eqref{eq:LS} we have
 $H_\Phi\subset \mathcal D (U)$, and uniform ellipticity is inherited under restriction to closed subspaces.
 Specifically, in case \eqref{eq:LS}, 
   for all $h\in\mathcal D(U)$, by definition there exist $u\in H^1_0(\Omega)$ and $\tau\in H(\mathrm{div})$ such that $h = (\nabla u,\tau,\mathrm{div}(\tau))$.
\end{proof}

Now let $\{ \Chi_n \colon n \in \mathcal N \}$ with $\mathcal N = \mathcal N_J$, be a frame for $H_\Phi$ and define 
\[
   \mathbf C_y = \bigl(  \langle \Chi_n, \Chi_{n'} \rangle_y  \bigr)_{n,n' \in \mathcal N}.
\]
We assume that $G \colon H_\Phi \to \ell_2(\mathcal N)$ is an injective mapping  that is uniformly in $J$ well-conditioned from $H_\Phi$ to $\ran G$ with the property
\[
      h = \sum_{n \in \mathcal N} (G h)_n \Chi_n \quad \text{for all $h \in H_\Phi$.}
\]
As noted earlier, $H_\Phi$ is a space of piecewise polynomials, so that $G$ can be obtained, for example, by piecewise polynomial interpolation.

Defining now
\[
  \mathbf D := G \mathcal D F,
\]
from the above definitions we directly obtain
\[
   \mathbf H^\intercal \mathbf A_y \mathbf H = \mathbf D^\intercal \mathbf C_y \mathbf D.
\]
The following is a direct consequence of Proposition \ref{prop:stab} and the properties of $G$. 

\begin{cor}
The condition numbers of $\mathbf D$ and $\mathbf C_y$ are bounded uniformly in $J$ and 
 $y\in Y$. 
\end{cor}

\begin{example}\label{ex:Edecomp}
Consider \eqref{eq:E} for $d=1$ with $\Omega = (0,1)$. Let $\varphi_{j,k}(x) = 2^{-j} \max\{ 0, 1 -  2^j \abs{x - k} \}$ for $k \in \mathcal{I}_j = \{ 1, \ldots, 2^j - 1\}$ for $j \in \N$, so that $\{ \varphi_{j,k}\colon j = 1,\ldots, J, \, k \in \mathcal{I}_j\}$ is a frame for $H^1_0(\Omega)$ with constants independent of $J$ as a consequence of Theorem \ref{thm:bpx}.
Let $\Chi_n$ be the characteristic function of the interval $\Omega_n = 2^{-J}(n-1, n)$ for $n \in \mathcal N_J = \{ 1, \ldots, 2^J\}$ and let $x_n = 2^{-J}(n + \frac12)$.
Then $\{\Chi_n : n\in \mathcal N_J\}$ is a frame for $H_\Phi$, which are piecewise constant functions on level $J$.
We thus obtain
\[
  \mathbf C_y = \biggl( \delta_{n,n'} \int_{\Omega_n} a(x,y) \,dx \biggr)_{n, n' \in \mathcal N_J}, \quad \mathbf D = \bigl( \varphi'_{jk}(x_n) \bigr)_{n \in \mathcal N_J, (j,k)}.
\]
For $d>1$ we can proceed analogously by choosing $\varphi_{j,k}$ as $H^1$-normalized piecewise linear finite elements on uniformly refined triangulations $\mathcal T_j$. Then we can choose $\{ \Chi_n \colon n \in \mathcal N_J\}$ as the basis of $\mathbb P_0(\mathcal T_J)^d$ formed from the characteristic functions of elements.
\end{example}

\begin{example}
For \eqref{eq:LS} in the case $d=1$ with $\Omega=(0,1)$, we have $H(\mathrm{div},\Omega) = H^1(\Omega)$. We can thus obtain a frame in $U = H^1_0(\Omega) \times H(\mathrm{div},\Omega) = H^1_0(\Omega)\times H^1(\Omega)$ by combining frames for the finite element spaces of level $J$ in $H^1_0(\Omega)$ and $H^1(\Omega)$ analogously to Example \eqref{ex:Edecomp}. Here, evaluation of the bilinear form $\langle \cdot, \cdot \rangle_y$ requires integrals of products of piecewise linear functions with the coefficients $a_y^{-1}$. In case that the latter are piecewise constant on the finest grid of level $J$, for example, these can be resolved using evaluation in two Gauss points per element and constructing the frame of $H_\Phi$ by combining vector-valued piecewise constant and piecewise linear functions.
In the case $d>1$, to the best of our knowledge, the construction of uniformly stable 
multilevel frames of $H(\mathrm{div},\Omega)$, that are applicable to spaces of Raviart-Thomas finite elements on hierarchies of triangular meshes, remains open. This question will be investigated in a separate work. 
\end{example}

\subsection{Network structures}\label{section:stable:structure}

With a multilevel frame $\Phi = \Phi_J$ as in Section \ref{sec:multilevel}, we consider approximations of parameter-dependent solutions the form \eqref{eq:hybrid}, 
\[
y \mapsto  \sum_{j=1}^J\sum_{k \in \mathcal{I}_j} \mathbf{u}_{jk}(y ; \theta) \varphi_{j,k},
\]
where the vector $\mathbf{u}(\cdot; \theta)$ is represented by a neural network with trainable parameters $\theta$.

We compare three different architectures,
illustrated in Figure~\ref{figure:network},
 for the neural networks mapping from the parameter space to the frame coefficients. A central element of all three architectures are the so-called \emph{residual networks} (ResNets)~\cite{He2016Deep}. These are widely used in scientific computing due to their favorable numerical stability properties. 

We use residual networks with layers $\Psi_{(A,W,b)} \colon \R^n\to\R^n$, with $n$ to be determined, of the particular form 
\begin{equation}\label{eq:resnetlayer}
	z \,\mapsto\, \Psi_{(A,W,b)}(z) = z+A\sigma( W z + b),
\end{equation}
where $A\in \mathbb R^{n \times r}$, $W\in \mathbb R^{r\times n}$, $b \in\mathbb R^n$, with componentwise application of the activation function $\sigma\colon \R\to\R$. 
We will also call such $\Psi_{(A,W,b)}$ a ResBlock.
Here, as in \cite{BachmayrDahmenOster2024}, we aim for $r \ll n$ to achieve a data-sparse representation, and by analogy to low-rank matrix representations, we will call $r$ the \emph{rank} of the layer.

\textbf{Full low-rank ResNet: } 
In the first architecture we use a fully connected layer \[  \R^{d_p}  \ni p \mapsto \Lambda_0(p) = \sigma(W_0 p + b_0) \in \R^{\hat N}, \] mapping to the full frame space followed by a sequence of ResBlocks, that is,
\[  \mathbf{u}( \cdot; \theta)=\Psi_{(A_{L}, W_{L}, b_{L})} \circ \cdots
\circ  \Psi_{(A_{2}, W_{2}, b_{2})} 
\circ \Psi_{(A_{1}, W_{1}, b_{1})} \circ \Lambda_0. \]

\textbf{Separate low-rank ResNet: } In the second architecture, we exploit the multilevel structure of the frame itself by using a distinct ResNet with initial fully connected layer for each frame level.
In the case \eqref{eq:E}, for the coefficients on each level $j$ we use a separate ResNet representation of the form
\[  \mathbf{u}_{j}( \cdot; \theta) = \Psi_{(A_{j,L_j}, W_{j,L_j}, b_{j,L_j})} \circ \cdots
\circ  \Psi_{(A_{j,2}, W_{j,2}, b_{j,2})} 
\circ \Psi_{(A_{j,1}, W_{j,1}, b_{j,1})} \circ \Lambda^j_0, \]
and $\mathbf u$ is obtained by concatenating the vectors $\mathbf u_j$ for $j \leq J$.
In~\eqref{eq:LS}, we parameterize the solution components $\sigma$ and $u$ by two independent separate ResNets.

\textbf{Separate frame representation: }
With $n_j = \#\mathcal{I}_j$, let the (sparse) prolongation matrices $P_j\in\mathbb R^{n_j \times n_{j-1}}$ be given for $j = 1,\ldots,J$ by
\begin{equation}\label{eq:prolongationdef}
   \sum_{k \in \mathcal{I}_{j-1}} \mathbf{v}_k \varphi_{j-1,k} =
   \sum_{i \in \mathcal{I}_{j}} (P_j \mathbf{v})_k \varphi_{j,k}
   , \quad \mathbf{v} \in  \R^{n_{j-1}}\,.
\end{equation}
We define \emph{prolongation layers} as $\Pi_{P_{j}} \colon \R^{n_{j-1}} \to \R^{n_j}, \, z\mapsto P_{j} z$.

Starting from a fully connected layer $\Lambda_0$ given by
\[  \R^{d_p}  \ni p \mapsto \Lambda_0(p) = \sigma(W_0 p + b_0) \in \R^{n_1}, \]
the resulting neural networks are composed of layers $\Lambda_j \colon \R^{n_j} \to \R^{n_j}$ of the form
\begin{equation}\label{eq:Lambdaj}
   \Lambda_j = \Psi_{(A_{j,L_j}, W_{j,L_j}, b_{j,L_j})} \circ \cdots
   \circ  \Psi_{(A_{j,2}, W_{j,2}, b_{j,2})} 
   \circ \Psi_{(A_{j,1}, W_{j,1}, b_{j,1})} \circ\Pi_{P_j}
\end{equation}
with some $L_j \in \N$, for $j = 1,\ldots,J$. 
In case \eqref{eq:E}, for each level $j$, the composition of such $\Lambda_0,\ldots, \Lambda_j$ parameterizes the frame coefficients $\mathbf u_j$ for this level,
\[
\mathbf{u}_{j}( \cdot; \theta_j)
= \Lambda_j \circ \cdots \circ \Lambda_1 \circ \Lambda_0,
\]
with trainable parameters that are separate for each $j$,
\begin{multline*}
   \theta_j = (A_{j,L_j}^{(j)}, W_{j,L_j}^{(j)}, b_{j,L_j}^{(j)}, \ldots, A_{j,1}^{(j)}, W_{j,1}^{(j)}, b_{j,1}^{(j)}, \ldots, \\ 
   \ldots, A_{1,L_1}^{(j)}, W_{1,L_1}^{(j)}, b_{1,L_1}^{(j)}, \ldots, A_{1,1}^{(j)}, W_{1,1}^{(j)}, b_{1,1}^{(j)}, W_0^{(j)}, b_0^{(j)})  .
\end{multline*}
This corresponds to a levelwise application of the format used previously in \cite{BachmayrDahmenOster2024}.
In~\eqref{eq:LS}, we again parameterize $\sigma$ and $u$ by two independent separate frame representations.

\begin{figure}

\begin{tikzpicture}[
  every node/.style={draw, align=center},
  input/.style={fill=black!30, line width=0.75},
  linear/.style={fill=orange!30, line width=0.75},
  resblock/.style={fill=teal!25, line width=0.75},
  output/.style={fill=blue!20, line width=0.75}
]

\node[input, minimum height=1.0cm, minimum width=0.25cm] (input) {};
\node[linear, minimum height=3.5cm, minimum width=0.25cm, right=1.0cm of input] (linear) {};
\node[resblock, minimum height=3.5cm, minimum width=0.25cm, right=0.5cm of linear] (r1) {};
\node[resblock, minimum height=3.5cm, minimum width=0.25cm, right=0.5cm of r1] (r2) {};
\node[resblock, minimum height=3.5cm, minimum width=0.25cm, right=0.5cm of r2] (r3) {};
\node[resblock, minimum height=3.5cm, minimum width=0.25cm, right=0.5cm of r3] (r4) {};
\node[resblock, minimum height=3.5cm, minimum width=0.25cm, right=0.5cm of r4] (r5) {};
\node[resblock, minimum height=3.5cm, minimum width=0.25cm, right=0.5cm of r5] (r6) {};
\node[resblock, minimum height=3.5cm, minimum width=0.25cm, right=0.5cm of r6] (r7) {};
\node[resblock, minimum height=3.5cm, minimum width=0.25cm, right=0.5cm of r7] (r8) {};

\draw[-Latex, line width=0.75] (input) -- (linear);
\draw[-Latex, line width=0.75] (linear) -- (r1);
\draw[-Latex, line width=0.75] (r1) -- (r2);
\draw[-Latex, line width=0.75] (r2) -- (r3);
\draw[-Latex, line width=0.75] (r3) -- (r4);
\draw[-Latex, line width=0.75] (r4) -- (r5);
\draw[-Latex, line width=0.75] (r5) -- (r6);
\draw[-Latex, line width=0.75] (r6) -- (r7);
\draw[-Latex, line width=0.75] (r7) -- (r8);

\node[output, minimum height=3.5cm, minimum width=0.25cm, right=8.25cm of r4] (output) {};

\draw[-Latex, line width=0.75] (r8) -- (output);
\end{tikzpicture}

\vspace{20pt}

\begin{tikzpicture}[
  every node/.style={draw, align=center},
  input/.style={fill=black!30, line width=0.75},
  linear/.style={fill=orange!30, line width=0.75},
  resblock/.style={fill=teal!25, line width=0.75},
  output/.style={fill=blue!20, line width=0.75}
]

\node[linear, minimum height=0.5cm, minimum width=0.25cm] (linear1) {};
\node[resblock, minimum height=0.5cm, minimum width=0.25cm, right=0.5cm of linear1] (r11) {};
\node[resblock, minimum height=0.5cm, minimum width=0.25cm, right=0.5cm of r11] (r12) {};
\node[resblock, minimum height=0.5cm, minimum width=0.25cm, right=0.5cm of r12] (r13) {};
\node[resblock, minimum height=0.5cm, minimum width=0.25cm, right=0.5cm of r13] (r14) {};
\node[resblock, minimum height=0.5cm, minimum width=0.25cm, right=0.5cm of r14] (r15) {};
\node[resblock, minimum height=0.5cm, minimum width=0.25cm, right=0.5cm of r15] (r16) {};
\node[resblock, minimum height=0.5cm, minimum width=0.25cm, right=0.5cm of r16] (r17) {};
\node[resblock, minimum height=0.5cm, minimum width=0.25cm, right=0.5cm of r17] (r18) {};
\draw[-Latex, line width=0.75] (linear1) -- (r11);
\draw[-Latex, line width=0.75] (r11) -- (r12);
\draw[-Latex, line width=0.75] (r12) -- (r13);
\draw[-Latex, line width=0.75] (r13) -- (r14);
\draw[-Latex, line width=0.75] (r14) -- (r15);
\draw[-Latex, line width=0.75] (r15) -- (r16);
\draw[-Latex, line width=0.75] (r16) -- (r17);
\draw[-Latex, line width=0.75] (r17) -- (r18);

\node[linear, minimum height=1.0cm, minimum width=0.25cm, below=0.0cm of linear1] (linear2) {};
\node[resblock, minimum height=1.0cm, minimum width=0.25cm, right=0.5cm of linear2] (r21) {};
\node[resblock, minimum height=1.0cm, minimum width=0.25cm, right=0.5cm of r21] (r22) {};
\node[resblock, minimum height=1.0cm, minimum width=0.25cm, right=0.5cm of r22] (r23) {};
\node[resblock, minimum height=1.0cm, minimum width=0.25cm, right=0.5cm of r23] (r24) {};
\node[resblock, minimum height=1.0cm, minimum width=0.25cm, right=0.5cm of r24] (r25) {};
\node[resblock, minimum height=1.0cm, minimum width=0.25cm, right=0.5cm of r25] (r26) {};
\node[resblock, minimum height=1.0cm, minimum width=0.25cm, right=0.5cm of r26] (r27) {};
\node[resblock, minimum height=1.0cm, minimum width=0.25cm, right=0.5cm of r27] (r28) {};
\draw[-Latex, line width=0.75] (linear2) -- (r21);
\draw[-Latex, line width=0.75] (r21) -- (r22);
\draw[-Latex, line width=0.75] (r22) -- (r23);
\draw[-Latex, line width=0.75] (r23) -- (r24);
\draw[-Latex, line width=0.75] (r24) -- (r25);
\draw[-Latex, line width=0.75] (r25) -- (r26);
\draw[-Latex, line width=0.75] (r26) -- (r27);
\draw[-Latex, line width=0.75] (r27) -- (r28);

\node[linear, minimum height=2.0cm, minimum width=0.25cm, below=0.0cm of linear2] (linear3) {};
\node[resblock, minimum height=2.0cm, minimum width=0.25cm, right=0.5cm of linear3] (r31) {};
\node[resblock, minimum height=2.0cm, minimum width=0.25cm, right=0.5cm of r31] (r32) {};
\node[resblock, minimum height=2.0cm, minimum width=0.25cm, right=0.5cm of r32] (r33) {};
\node[resblock, minimum height=2.0cm, minimum width=0.25cm, right=0.5cm of r33] (r34) {};
\node[resblock, minimum height=2.0cm, minimum width=0.25cm, right=0.5cm of r34] (r35) {};
\node[resblock, minimum height=2.0cm, minimum width=0.25cm, right=0.5cm of r35] (r36) {};
\node[resblock, minimum height=2.0cm, minimum width=0.25cm, right=0.5cm of r36] (r37) {};
\node[resblock, minimum height=2.0cm, minimum width=0.25cm, right=0.5cm of r37] (r38) {};
\draw[-Latex, line width=0.75] (linear3) -- (r31);
\draw[-Latex, line width=0.75] (r31) -- (r32);
\draw[-Latex, line width=0.75] (r32) -- (r33);
\draw[-Latex, line width=0.75] (r33) -- (r34);
\draw[-Latex, line width=0.75] (r34) -- (r35);
\draw[-Latex, line width=0.75] (r35) -- (r36);
\draw[-Latex, line width=0.75] (r36) -- (r37);
\draw[-Latex, line width=0.75] (r37) -- (r38);

\node[output, minimum height=2.0cm, minimum width=0.25cm, right=8.25cm of r34] (output3) {};
\node[output, minimum height=1.0cm, minimum width=0.25cm, above=0.0cm of output3] (output2) {};
\node[output, minimum height=0.5cm, minimum width=0.25cm, above=0.0cm of output2] (output1) {};

\draw[-Latex, line width=0.75] (r18) -- (output1);
\draw[-Latex, line width=0.75] (r28) -- (output2);
\draw[-Latex, line width=0.75] (r38) -- (output3);

\node[input, minimum height=1.0cm, minimum width=0.25cm, left=1.0cm of linear2] (input) {};
\draw[-, line width=0.75] (input.east) -- ([xshift=-0.5cm]linear2.west);
\draw[-Latex, line width=0.75] ([xshift=-0.5cm]linear2.west) |- (linear1.west);
\draw[-Latex, line width=0.75] ([xshift=-0.5cm]linear2.west) -- (linear2.west);
\draw[-Latex, line width=0.75] ([xshift=-0.5cm]linear2.west) |- (linear3.west);

\end{tikzpicture}

\vspace{20pt}

\begin{tikzpicture}[
  every node/.style={draw, align=center},
  input/.style={fill=black!30, line width=0.75},
  linear/.style={fill=orange!30, line width=0.75},
  resblock/.style={fill=teal!25, line width=0.75},
  output/.style={fill=blue!20, line width=0.75}
]

\tikzset{
  pics/prolongation1/.style={
    code={
      \def\L{0}        %
      \def\R{0.25}      %
      \def\Hl{0.25}     %
      \def\Hr{0.5}    %
      \draw[fill=red!25, draw, line width=0.75]
        (\L,-\Hl) --
        (\L, \Hl) --
        (\R, \Hr) --
        (\R,-\Hr) -- cycle;
      \path (\L,0) coordinate (-west);
      \path (\R,0) coordinate (-east);
      \path ({(\L+\R)/2},0) coordinate (-center);
    }
  }
}

\tikzset{
  pics/prolongation2/.style={
    code={
      \def\L{0}        %
      \def\R{0.25}      %
      \def\Hl{0.5}     %
      \def\Hr{1.0}    %
      \draw[fill=red!25, draw, line width=0.75]
        (\L,-\Hl) --
        (\L, \Hl) --
        (\R, \Hr) --
        (\R,-\Hr) -- cycle;
      \path (\L,0) coordinate (-west);
      \path (\R,0) coordinate (-east);
      \path ({(\L+\R)/2},0) coordinate (-center);
    }
  }
}

\node[linear, minimum height=0.5cm, minimum width=0.25cm] (linear1) {};
\node[resblock, minimum height=0.5cm, minimum width=0.25cm, right=0.5cm of linear1] (r11) {};
\node[resblock, minimum height=0.5cm, minimum width=0.25cm, right=0.5cm of r11] (r12) {};
\node[resblock, minimum height=0.5cm, minimum width=0.25cm, right=0.5cm of r12] (r13) {};
\node[resblock, minimum height=0.5cm, minimum width=0.25cm, right=0.5cm of r13] (r14) {};

\draw[-Latex, line width=0.75] (linear1) -- (r11);
\draw[-Latex, line width=0.75] (r11) -- (r12);
\draw[-Latex, line width=0.75] (r12) -- (r13);
\draw[-Latex, line width=0.75] (r13) -- (r14);

\node[linear, minimum height=0.5cm, minimum width=0.25cm, below=0.25cm of linear1] (linear2) {};
\node[resblock, minimum height=0.5cm, minimum width=0.25cm, right=0.5cm of linear2] (r21) {};
\node[resblock, minimum height=0.5cm, minimum width=0.25cm, right=0.5cm of r21] (r22) {};
\node[resblock, minimum height=0.5cm, minimum width=0.25cm, right=0.5cm of r22] (r23) {};
\node[resblock, minimum height=0.5cm, minimum width=0.25cm, right=0.5cm of r23] (r24) {};
\pic[right=0.5cm of r24] (p21) {prolongation1};
\node[resblock, minimum height=1.0cm, minimum width=0.25cm, right=0.5cm of p21-east] (r25) {};
\node[resblock, minimum height=1.0cm, minimum width=0.25cm, right=0.5cm of r25] (r26) {};
\node[resblock, minimum height=1.0cm, minimum width=0.25cm, right=0.5cm of r26] (r27) {};
\node[resblock, minimum height=1.0cm, minimum width=0.25cm, right=0.5cm of r27] (r28) {};

\draw[-Latex, line width=0.75] (linear2) -- (r21);
\draw[-Latex, line width=0.75] (r21) -- (r22);
\draw[-Latex, line width=0.75] (r22) -- (r23);
\draw[-Latex, line width=0.75] (r23) -- (r24);
\draw[-Latex, line width=0.75] (r24) -- (p21-west);
\draw[-Latex, line width=0.75] (p21-east) -- (r25);
\draw[-Latex, line width=0.75] (r25) -- (r26);
\draw[-Latex, line width=0.75] (r26) -- (r27);
\draw[-Latex, line width=0.75] (r27) -- (r28);

\node[linear, minimum height=0.5cm, minimum width=0.25cm, below=1.0cm of linear2] (linear3) {};
\node[resblock, minimum height=0.5cm, minimum width=0.25cm, right=0.5cm of linear3] (r31) {};
\node[resblock, minimum height=0.5cm, minimum width=0.25cm, right=0.5cm of r31] (r32) {};
\node[resblock, minimum height=0.5cm, minimum width=0.25cm, right=0.5cm of r32] (r33) {};
\node[resblock, minimum height=0.5cm, minimum width=0.25cm, right=0.5cm of r33] (r34) {};
\pic[right=0.5cm of r34] (p31) {prolongation1};
\node[resblock, minimum height=1.0cm, minimum width=0.25cm, right=0.5cm of p31-east] (r35) {};
\node[resblock, minimum height=1.0cm, minimum width=0.25cm, right=0.5cm of r35] (r36) {};
\node[resblock, minimum height=1.0cm, minimum width=0.25cm, right=0.5cm of r36] (r37) {};
\node[resblock, minimum height=1.0cm, minimum width=0.25cm, right=0.5cm of r37] (r38) {};
\pic[right=0.5cm of r38] (p32) {prolongation2};
\node[resblock, minimum height=2.0cm, minimum width=0.25cm, right=0.5cm of p32-east] (r39) {};
\node[resblock, minimum height=2.0cm, minimum width=0.25cm, right=0.5cm of r39] (r310) {};
\node[resblock, minimum height=2.0cm, minimum width=0.25cm, right=0.5cm of r310] (r311) {};
\node[resblock, minimum height=2.0cm, minimum width=0.25cm, right=0.5cm of r311] (r312) {};

\draw[-Latex, line width=0.75] (linear3) -- (r31);
\draw[-Latex, line width=0.75] (r31) -- (r32);
\draw[-Latex, line width=0.75] (r32) -- (r33);
\draw[-Latex, line width=0.75] (r33) -- (r34);
\draw[-Latex, line width=0.75] (r34) -- (p31-west);
\draw[-Latex, line width=0.75] (p31-east) -- (r35);
\draw[-Latex, line width=0.75] (r35) -- (r36);
\draw[-Latex, line width=0.75] (r36) -- (r37);
\draw[-Latex, line width=0.75] (r37) -- (r38);
\draw[-Latex, line width=0.75] (r38) -- (p32-west);
\draw[-Latex, line width=0.75] (p32-east) -- (r39);
\draw[-Latex, line width=0.75] (r39) -- (r310);
\draw[-Latex, line width=0.75] (r310) -- (r311);
\draw[-Latex, line width=0.75] (r311) -- (r312);

\node[output, minimum height=2.0cm, minimum width=0.25cm, right=0.5cm of r312] (output3) {};
\node[output, minimum height=1.0cm, minimum width=0.25cm, above=0.0cm of output3] (output2) {};
\node[output, minimum height=0.5cm, minimum width=0.25cm, above=0.0cm of output2] (output1) {};

\draw[-Latex, line width=0.75] (r14) -- (output1);
\draw[-Latex, line width=0.75] (r28) -- (output2);
\draw[-Latex, line width=0.75] (r312) -- (output3);

\node[input, minimum height=1.0cm, minimum width=0.25cm, left=1.0cm of linear2] (input) {};
\draw[-, line width=0.75] (input.east) -- ([xshift=-0.5cm]linear2.west);
\draw[-Latex, line width=0.75] ([xshift=-0.5cm]linear2.west) |- (linear1.west);
\draw[-Latex, line width=0.75] ([xshift=-0.5cm]linear2.west) -- (linear2.west);
\draw[-Latex, line width=0.75] ([xshift=-0.5cm]linear2.west) |- (linear3.west);

\end{tikzpicture}

\vspace{10pt}

\begin{tikzpicture}[
every node/.style={draw, align=center},
input/.style={fill=black!30, line width=0.75},
linear/.style={fill=orange!30, line width=0.75},
resblock/.style={fill=teal!25, line width=0.75},
output/.style={fill=blue!20, line width=0.75}
]

\node[input, minimum height=0.35cm, minimum width=0.35cm] (i) {};
\node[right=0.1cm of i, draw=none] {Input parameters};
\node[output, minimum height=0.35cm, minimum width=0.35cm, right=4.0cm of i] (o) {};
\node[right=0.1cm of o, draw=none] {Output frame coefficients};

\node[linear, minimum height=0.35cm, minimum width=0.35cm, below=0.25cm of i] (l) {};
\node[right=0.1cm of l, draw=none] {Linear layer};
\node[resblock, minimum height=0.35cm, minimum width=0.35cm, right=4.0cm of l] (r) {};
\node[right=0.1cm of r, draw=none] {Low-rank ResBlock};
\node[fill=red!25, minimum height=0.35cm, minimum width=0.35cm, right=4.0cm of r, line width=0.75] (p) {};
\node[right=0.1cm of p, draw=none] {Prolongation};

\end{tikzpicture}

\caption{Different neural network architectures mapping from the parameter space to the frame coefficients. Without loss of generality, we take three-levels of frame as an example. \textbf{(Top)} Full low-rank ResNet.  \textbf{(Middle)} Separate low-rank ResNet. \textbf{(Bottom)} Separate frame representation.}
\label{figure:network}
\end{figure}
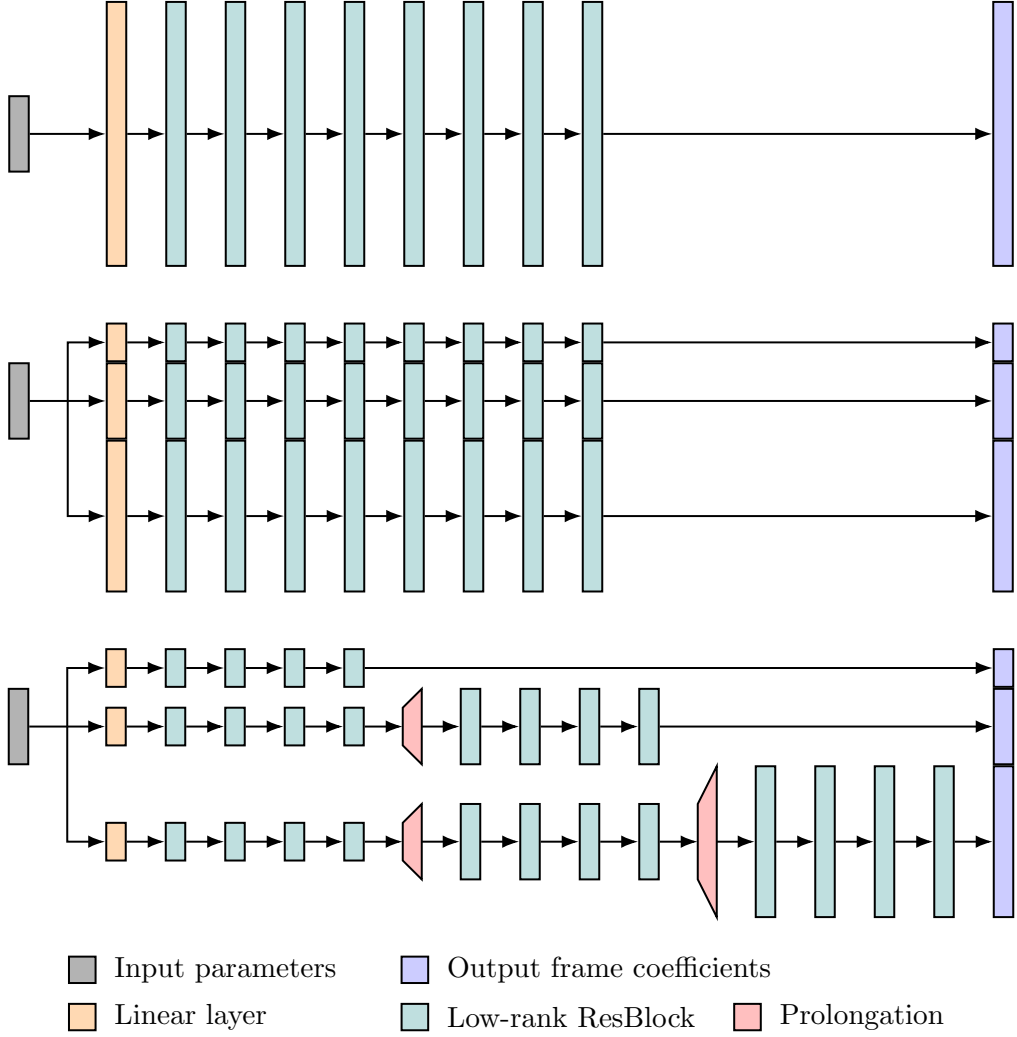

\subsection{Matrix-free numerical realization}

In the variational formulation~\eqref{eq:residualregression}, instead of assembling the stiffness matrix $\mathbf{A}_{y}$ explicitly, we employ a matrix-free evaluation of matrix-vector products. 
We again assume a multilevel frame $\Phi = \Phi_J$ as in Section \ref{sec:multilevel} corresponding to the mesh $\mathcal T_J$ on level $J$.
In particular, for $j=1,\ldots,J$, we have a basis $\{ \varphi_{j,k} \colon k \in \mathcal{I}_j \}$ of the respective finite element space on level $j$. 

We consider the variational formulation on the mesh $\mathcal T_J = \{ T_1, \ldots, T_{\# \mathcal T_J} \}$. 
For each element $T_k$, the local-to-global mapping is given by an index set 
\[
I_{k}\coloneq\{i_{k}^{1},i_{k}^{2},\ldots,i_{k}^{m}\},
\]
where $m$ denotes the number of local degrees of freedom (e.g., $m=2$ for one-dimensional linear elements and $m=3$ for triangular elements in two dimensions). 
On this element, we define the local stiffness matrix
\[
\mathbf{A}_{y}^{k}\coloneq 
\bigl( \langle \mathcal{D}\varphi_{J,\ell},  \mathcal{D}\varphi_{J,\ell^{\prime}} \rangle_{y} \bigr)_{\ell,\ell^{\prime}\in I_{k}}
\in\mathbb{R}^{m\times m},
\]
where $\mathcal{D}$ denotes the differential operator associated with the variational formulation.

Given a global solution vector $\mathbf{u}\in\mathbb{R}^{\#\mathcal{I}_{J}}$, we gather the corresponding local subvector
\[
\mathbf{u}^{k} \coloneq (u_{i_{k}^{1}},u_{i_{k}^{2}},\ldots,u_{i_{k}^{m}})^{\intercal}\in\mathbb{R}^{m}.
\]
The local contributions are then assembled into the global vector via
\[
\bigl(\mathbf{A}_{y}\mathbf{u}\bigr)_{i}
=\sum_{k: i\in I_{k}}
\bigl(\mathbf{A}_{y}^{k}\mathbf{u}^{k}\bigr)_{\mathrm{loc}(i,T_{k})},
\]
where $\mathrm{loc}(i,T_{k})$ denotes the local index of the global degree of freedom $i$ within element $T_{k}$, i.e., $i$ is the $\mathrm{loc}(i,T_{k})$-th entry of $I_{k}$. 
This matrix-vector product allows us to evaluate the variational formulation~\eqref{eq:residualregression} without explicitly forming the global stiffness matrix.

Moreover, the elementwise products $\{\mathbf{A}_{y}^{k}\mathbf{u}^{k}\}_{k\in\mathcal{I}_{J}}$ can be computed in a vectorized manner. 
Specifically, 
\[
\widetilde{\mathbf{A}}_{y}
\coloneq 
\begin{bmatrix}
\mathbf{A}_{y}^{1} \\
\vdots \\
\mathbf{A}_{y}^{\#\mathcal{I}_{J}}
\end{bmatrix}\in\mathbb{R}^{\#\mathcal{I}_{J}\times m\times m},
\qquad
\widetilde{\mathbf{u}}
\coloneq
\begin{bmatrix}
\mathbf{u}^{1} \\ \vdots \\ \mathbf{u}^{\#\mathcal{I}_{J}}
\end{bmatrix}\in\mathbb{R}^{\#\mathcal{I}_{J}\times m},
\]
such that the product in each element can be computed all-at-once by a tensor contraction:
\[
\bigl(\widetilde{\mathbf{A}}_{y}\bigr)_{k\alpha\beta}\bigl(\widetilde{\mathbf{u}}\bigr)_{k\alpha}
=
\begin{bmatrix}
\mathbf{A}_{y}^{1}\mathbf{u}^{1} \\ \vdots \\ \mathbf{A}_{y}^{\#\mathcal{I}_{J}}\mathbf{u}^{\#\mathcal{I}_{J}}
\end{bmatrix}\in\mathbb{R}^{\#\mathcal{I}_{J}\times m}.
\]
This formulation naturally extends to batched parameters $y$, enabling simultaneous evaluation of stiffness matrix-vector products for multiple parameter instances.
Consequently, the matrix-free assembly procedure significantly reduces memory consumption and replaces sparse matrix-vector products by a batched dense tensor contraction, which are more efficient on GPUs.

For evaluating the product $\mathbf{H}\mathbf{w}$ for a given vector $\mathbf{w}$, we utilize the standard recursive implementation of multilevel preconditioners.
Combining this with the above routines, yields a matrix-free numerical realization of the action of $\mathbf H^\intercal \mathbf A_y \mathbf H$.

Finally, we now turn to the matrix-free realization of the numerically stable decomposition $\mathbf D^\intercal \mathbf C_y \mathbf D$.
Let $w_{n}^{q},x_{n}^{q}$ for $q=1,\dots, Q_n$ be Gauss quadrature weights and points corresponding to the triangle $T_{n}$ such that quadratic functions are exactly integrated on the triangles. Suppose $a_y$ is piecewise constant on this triangulation.

For a given coefficient vector $\mathbf{w}$, we introduce its subdivision into level components,
\begin{equation*}
\mathbf{w}\coloneq (\mathbf{w}_{1},\ldots,\mathbf{w}_{J}), \quad \mathbf{w}_{j}\in\mathbb{R}^{\#\mathcal{I}_{j}}.
\end{equation*}
Then 
\begin{equation*}
\mathcal{D}F(\mathbf{w})=\mathcal{D}\Bigg(\sum_{j=1}^{J}\sum_{k\in\mathcal{I}_{j}}w_{j,k}\varphi_{j,k}\Bigg)=\sum_{j=1}^{J}\sum_{k\in\mathcal{I}_{j}}\mathbf{w}_{j,k}\mathcal{D}\varphi_{j,k}.
\end{equation*}
Consequently,
\begin{align*}
\langle \mathcal{D}F(\mathbf{w}),\mathcal{D}F(\mathbf{w}) \rangle_{y}
&=\sum_{j,j^{\prime}=1}^{J}\sum_{k,k^{\prime}\in\mathcal{I}_{j}}\mathbf{w}_{j,k}\mathbf{w}_{j^{\prime},k^{\prime}}\langle\mathcal{D}\varphi_{j,k},\mathcal{D}\varphi_{j^{\prime},k^{\prime}}\rangle_{y} \\
&=\sum_{j,j^{\prime}=1}^{J}\sum_{k,k^{\prime}\in\mathcal{I}_{j}}\mathbf{w}_{j,k}\mathbf{w}_{j^{\prime},k^{\prime}} \sum_{n=1}^{N_{J}}\sum_{q=1}^{Q}w_{n}^{q}a_{y}(x_{n}^{q})\langle\mathcal{D}\varphi_{j,k}(x_{n}^{q}),\mathcal{D}\varphi_{j^{\prime},k^{\prime}}(x_{n}^{q})\rangle_{\ell_{2}}.
\end{align*}
Therefore, the $(q,n)$-th row of the matrix $\mathbf{D}$ reads
\begin{equation*}
\mathbf{D}_{q,n} \coloneq 
\begin{bmatrix}
\underbrace{\mathcal{D}\varphi_{1,1}(x_{n}^{q}),\ldots,\mathcal{D}\varphi_{1,\#\mathcal{I}_{1}}(x_{n}^{q})}_{\text{level 1}} & \cdots & \underbrace{\mathcal{D}\varphi_{J,1}(x_{n}^{q}),\ldots,\mathcal{D}\varphi_{J,\#\mathcal{I}_{J}}(x_{n}^{q})}_{\text{level $J$}}
\end{bmatrix}.
\end{equation*}

In practice, it is not necessary to precompute or store the sparse matrix $\mathbf{D}$. Instead, one can evaluate the vector inner products on the fly,
\begin{equation*}
\langle \mathbf{D}_{q,n}, \mathbf{w} \rangle_{\ell_{2}}
= \sum_{j=1}^{J} \sum_{k \in \mathcal{I}_{j}}
\Big\langle 
\underbrace{(\mathcal{D}\varphi_{j,k}(x_{n}^{q}), \ldots, \mathcal{D}\varphi_{j,\#\mathcal{I}_{j}}(x_{n}^{q}))}_{\mathbf{D}_{q,n,j}},
\mathbf{w}_{j}
\Big\rangle_{\ell_{2}},
\end{equation*}
which can be computed efficiently. Specifically, suppose that the quadrature point $x_{n}^{q}$ lies in an element $T_{j,k^{*}}$ at level $j$. The associated local-to-global index mapping is given by
\begin{equation*}
I_{j,k^{*}} \coloneq \{ i_{k^{*}}^{1}, i_{k^{*}}^{2}, \ldots, i_{k^{*}}^{m} \}.
\end{equation*}
We gather the corresponding local subvectors of $\mathbf{D}_{q,n,j}$ and $\mathbf{w}_{j}$, and compute their inner product. To vectorize this procedure, we apply the gather operator over all $(q,n)$ simultaneously. A subsequent tensor contraction then yields the desired product $\mathbf{D}\mathbf{w}$.

\subsection{Experimental settings}\label{sec:exp}
In our tests, we restrict ourselves to a one-dimensional parametric elliptic equation, since the effects that are relevant for this work do not differ  in higher-dimensional cases. The diffusion field $a_{y}$ is defined as 
\begin{equation*}
a_{y}=y_{1}\chi_{[0,1/4)}+y_{2}\chi_{[1/4,1/2)}+y_{3}\chi_{[1/2,3/4)}+y_{4}\chi_{[3/4,1]},
\end{equation*}
where $y \coloneq (y_{1},y_{2},y_{3},y_{4}) \in[0.5, 1.5]^{4}$, and $\chi_S$ denotes the characteristic function of the set $S$. The source term is set as a constant function $f\equiv 1$.

\textbf{Metrics.}
We use three metrics to evaluate the accuracy of the solutions $\mathbf u$ predicted by neural network-based surrogate models, compared to the direct computation $\mathbf u_\mathrm{FE}$ of the finite element solution for the given parameter values. The first metric is the mean relative error (MRE), defined as
\begin{equation*}
\mathrm{MRE} \coloneq \frac{1}{N_{\rm test}}\sum_{i=1}^{N_{\rm test}}
\frac{\|\mathbf{F}\mathbf{u}(y_{i})-\mathbf{u}_{\rm FE}(y_{i})\|_{2}}
{\|\mathbf{u}_{\rm FE}(y_{i})\|_{2}}.
\end{equation*}
It is worth noting that for the vanilla FOSLS formulation, $\mathbf{F}=\mathrm{id}$, whereas for FOSLS with frame preconditioning, $\mathbf{F}=\mathbf{H}^{\intercal}$ denotes the frame synthesis operator. The second metric is the mean squared error (MSE), given by
\begin{equation*}
\mathrm{MSE} \coloneq \frac{1}{N_{\rm test}}\sum_{i=1}^{N_{\rm test}}
\|\mathbf{F}\mathbf{u}(y_{i})-\mathbf{u}_{\rm FE}(y_{i})\|_{2}^{2}.
\end{equation*}
The third metric is the value of the FOSLS loss function~\eqref{eq:residualregression} in double precision evaluated on the test dataset. For frame coefficients predicted by the neural network, we first apply the frame synthesis operator and then substitute the reconstructed solution into the standard FOSLS loss. As shown in~\eqref{eq:LSreserr}, this loss provides a quantitative measure of the solution error. Evaluating the loss on the test dataset avoids the need to solve a linear system within the finite element method and thus provides a computationally inexpensive yet reliable performance metric. 

\textbf{Frame preconditioning.}
We employ $J=10$ frame levels. The Lagrange basis functions are normalized with respect to the $H^{1}$-norm, which is computed using two-point Gaussian quadrature. When using half-precision arithmetic, in practice, we observe numerical instability when computing these $H^{1}$-norms. To address this issue, we precompute the $H^{1}$-norms in double precision and subsequently convert them to half or single precision for use during training in half- or single-precision.

\textbf{Neural network architectures.}
We use the sigmoid linear unit (SiLU)~\cite{Elfwing2018Sigmoid,ramachandran2017Swish} activation function, and all network parameters are initialized using the Xavier Gaussian distribution~\cite{Glorot2010Understanding}.
\begin{enumerate}[(i)]
\item \emph{Full low-rank ResNet.}
For the full ResNet architecture, we employ 8 ResBlocks with a fixed rank of 8.
\item \emph{Separate low-rank ResNet.}
In the FOSLS setting, we parameterize $\sigma$ and $u$ using two independent ResNets. Each ResNet consists of 10 sub-ResNets, where each sub-ResNet maps the parameter space to the frame coefficients at a specific level. Every sub-ResNet contains 8 ResBlocks, and the rank of each ResBlock is set to the minimum of 8 and the input dimension of the ResBlock.
\item \emph{Separate frame representations.}
In the FOSLS framework, we also consider parameterizing $\sigma$ and $u$ using two independent frame-based representations. Between two consecutive prolongation operators, we insert 4 ResBlocks, i.e., $L_j = 4$ for each $j = 1,\ldots,J$ in~\eqref{eq:Lambdaj}. The rank of the ResBlocks is chosen as the minimum of 8 and the input dimension of the ResBlock.
\end{enumerate}

\textbf{Hyper-parameters for training.} 
We employ several optimizers depending on the experimental setup, with hyperparameters adjusted according to the float precision of computation (\texttt{float16}, \texttt{float32}, or \texttt{float64}). We use 6000 epochs for each optimizer.

\begin{enumerate}[(i)]
\item \emph{Adam.} For the Adam optimizer, the initial learning rate $\eta_{0}$ is set to $10^{-3}$. We apply a cosine annealing learning rate schedule
\[
\eta_t = \eta_{\min} + \frac{1}{2} (\eta_0 - \eta_{\min}) \left( 1 + \cos \frac{t \pi}{T_{\rm max}} \right)
\]
with period $T_{\rm max} = 5000$, and the minimum learning rate is set according to precision: $10^{-4}$ for half precision, $10^{-6}$ for single precision, and $10^{-10}$ for double precision. The parameter $\epsilon$ in Adam is chosen to ensure numerical stability: $10^{-4}$ for half precision, $10^{-8}$ for single precision, and $10^{-16}$ for double precision.
\item \emph{SGD.} 
The standard SGD is combined with a cosine annealing learning rate schedule, whose hyperparameters are set as those of Adam.
\item \emph{L-BFGS.} 
For the L-BFGS optimizer, we set the learning rate to $1.0$, with a maximum of 20 iterations and 25 function evaluations per step. The gradient and parameter change tolerances are set according to the precision: for half precision, both are $10^{-4}$; for single precision, $10^{-8}$; and for double precision, $10^{-12}$. The history size is fixed at 100, and the line search uses the strong Wolfe condition.
\item\emph{Natural Gradient Descent (NGD).}
For NGD, we use an initial step size of $1.0$, combined with a line search satisfying the strong Wolfe conditions and allowing at most 20 line search steps. The inner linear system is solved using the conjugate gradient method without explicitly forming the system matrix. The conjugate gradient solver is run for at most 20 iterations, and both the conjugate gradient tolerance and the $\epsilon$ regularization parameter depend on the numerical precision: $10^{-9}$ for half precision, $10^{-12}$ for single precision, and $10^{-14}$ for double precision.

\end{enumerate}

\subsection{Numerically stable training}

As reported in Table \ref{tab:with:without:preconditioning}, preconditioning leads to considerable improvements in the performance of neural network training even using the numerically unstable representation $\mathbf H^\intercal \mathbf A_y \mathbf H$ of the resulting system matrix. The corresponding loss curves are shown in Figure \ref{fig:fosls:with:without:preconditioning}. However, the advantage of four orders of magnitude in the loss value is realized by the Adam optimizer, but not by SGD.

\begin{figure}[htbp]
\centering
\includegraphics[width=1.0\linewidth]{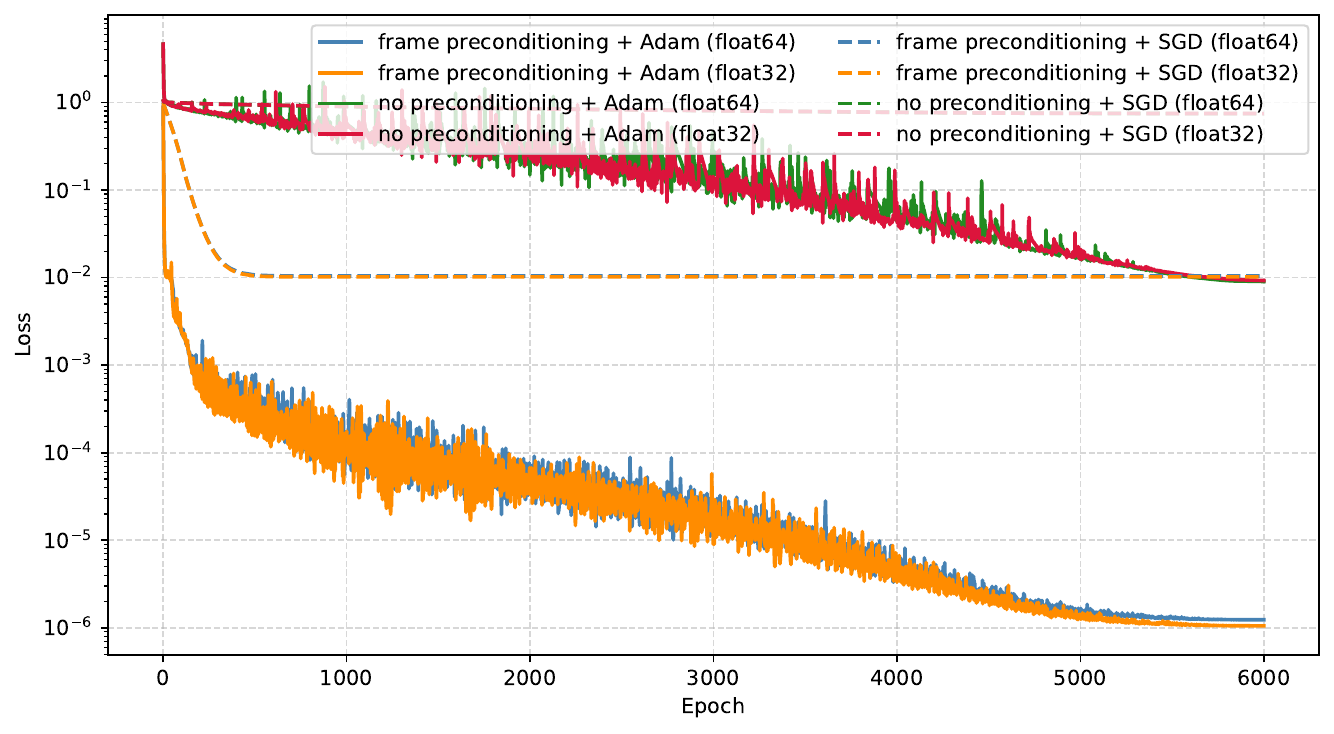}
\caption{Loss curves of FOSLS with and without preconditioning using different optimizers under different computing precisions.}
\label{fig:fosls:with:without:preconditioning}
\end{figure}

With the aid of stable frame representations, we can apply frame-based preconditioning also in half precision (using 16-bit floating point numbers), as reported in Table~\ref{tab:stable:preconditioning}. Comparing Table~\ref{tab:stable:preconditioning} with Table~\ref{tab:with:without:preconditioning}, we observe that the accuracy is comparable in single and double precision. In half precision, the loss value achieved by Adam reaches $10^{-4}$, which is close to the machine precision of half-precision floating-point arithmetic. Although L-BFGS and NGD perform well in single and double precision, their performance in half precision is inferior to that of Adam. In practice, we find that L-BFGS and NGD tend to perform better on shallower networks with larger rank when using half precision. It is worth noting that L-BFGS and NGD are more sensitive to numerical precision: their performance differs significantly between single and double precision, whereas the differences for SGD and Adam are relatively minor. The corresponding loss curves are shown in Figure~\ref{fig:fosls:stable:preconditioning}.

\begin{table}[htbp]
\centering
\caption{Comparison of FOSLS with stable preconditioning using different optimizers under different computing precisions.}
\label{tab:stable:preconditioning}
\small
\setlength{\tabcolsep}{4pt}
\begin{tabular}{lcccc}
\toprule
Optimizer & Precision & MRE & MSE & Loss \\
\midrule 
SGD & \texttt{float16} & $1.85\times 10^{-1}$ & $3.83$ & $1.86\times 10^{-2}$ \\
 & \texttt{float32} & $1.45\times 10^{-1}$ & $2.62$ & $1.06\times 10^{-2}$ \\
 & \texttt{float64} & $1.44\times 10^{-1}$ & $2.60$ & $1.06\times 10^{-2}$ \\
\midrule
Adam & \texttt{float16} & $7.25\times 10^{-3}$ & $6.73\times 10^{-3}$ & $1.02\times 10^{-4}$ \\
 & \texttt{float32} & $1.01\times 10^{-3}$ & $1.52\times 10^{-4}$ & $1.20\times 10^{-6}$ \\
 & \texttt{float64} & $1.26\times 10^{-3}$ & $2.19\times 10^{-4}$ & $1.14\times 10^{-6}$ \\
\midrule
L-BFGS & \texttt{float16} & $1.47\times 10^{-1}$ & $2.70$ & $1.21\times 10^{-2}$ \\
 & \texttt{float32} & $5.93\times 10^{-3}$ & $6.02\times 10^{-3}$ & $4.84\times 10^{-5}$ \\
 & \texttt{float64} & $9.46\times 10^{-4}$ & $1.36\times 10^{-4}$ & $1.14\times 10^{-6}$ \\
\midrule
NGD & \texttt{float16} & $3.34\times 10^{-2}$ & $1.58\times 10^{-1}$ & $1.73\times 10^{-3}$ \\
 & \texttt{float32} & $3.29\times 10^{-3}$ & $1.80\times 10^{-3}$ & $1.05\times 10^{-5}$ \\
 & \texttt{float64} & $2.62\times 10^{-3}$ & $1.26\times 10^{-3}$ & $5.08\times 10^{-6}$ \\
\bottomrule
\end{tabular}
\end{table}

\begin{figure}[htbp]
\centering
\includegraphics[width=1.0\linewidth]{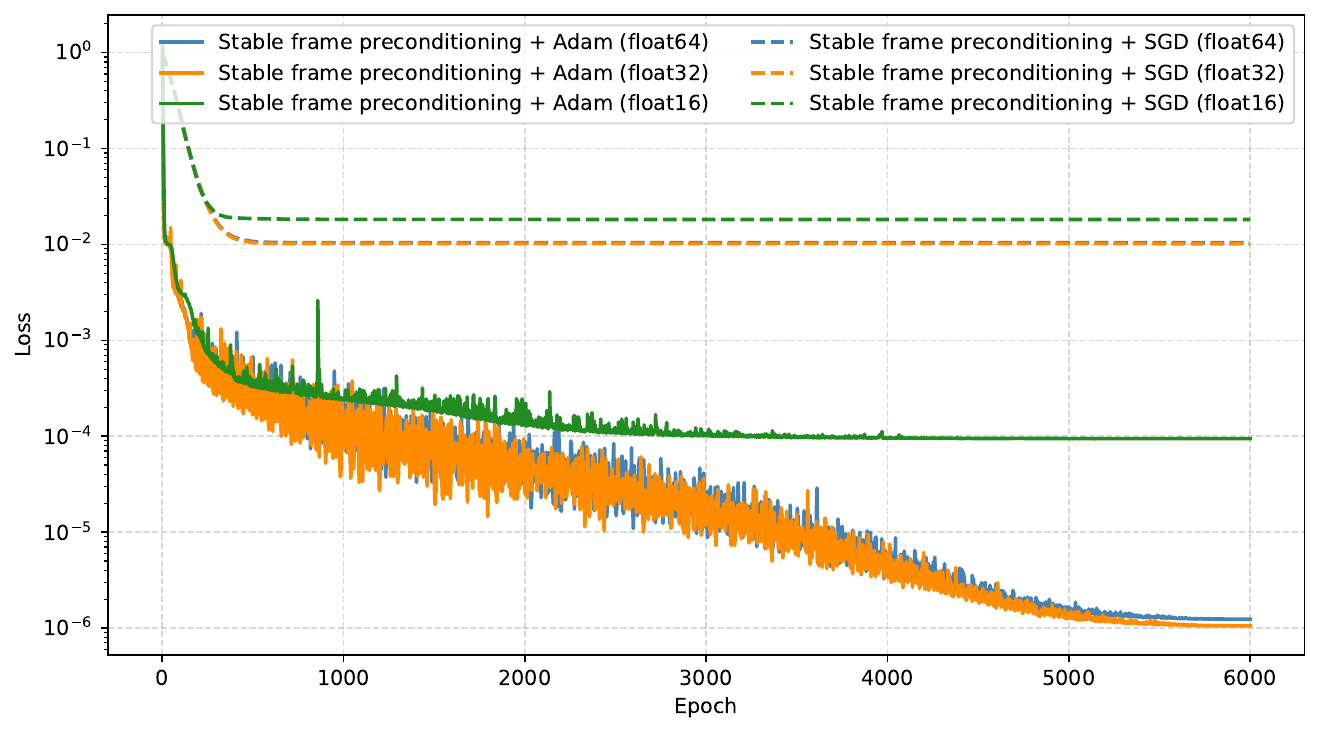}
\caption{Loss curves of numerically stable preconditioning using different optimizers under different computing precisions.}
\label{fig:fosls:stable:preconditioning}
\end{figure}

\par We further compare different network architectures introduced in Section~\ref{section:stable:structure}, as summarized in Table~\ref{tab:stable:preconditioning:sep}. The three neural networks have nearly the same number of trainable parameters (degrees of freedom, DoFs). All three architectures exhibit similar performance across the three numerical precisions.

\begin{table}[htbp]
\centering
\caption{Comparison of FOSLS with stable preconditioning using different neural network architectures.}
\label{tab:stable:preconditioning:sep}
\small
\setlength{\tabcolsep}{4pt}
\begin{tabular}{lccccc}
\toprule
Architecture & DoFs & Precision & MRE & MSE & Loss \\
\midrule 
Full ResNets & 577036 & \texttt{float16} & $7.25\times 10^{-3}$ & $6.73\times 10^{-3}$ & $1.02\times 10^{-4}$ \\
 & & \texttt{float32} & $1.01\times 10^{-3}$ & $1.52\times 10^{-4}$ & $1.20\times 10^{-6}$ \\
 & & \texttt{float64} & $1.26\times 10^{-3}$ & $2.19\times 10^{-4}$ & $1.14\times 10^{-6}$ \\
\midrule 
Separate ResNets & 577140 & \texttt{float16} & $8.83\times 10^{-3}$ & $1.10\times 10^{-2}$ & $1.65\times 10^{-4}$ \\
 & & \texttt{float32} & $3.09\times 10^{-3}$ & $1.52\times 10^{-3}$ & $2.46\times 10^{-6}$ \\
 & & \texttt{float64} & $3.59\times 10^{-3}$ & $2.11\times 10^{-3}$ & $3.52\times 10^{-6}$ \\
\midrule
Separate frame rep. & 551336 & \texttt{float16} & $1.83\times 10^{-2}$ & $5.12\times 10^{-2}$ & $3.29\times 10^{-4}$ \\
 & & \texttt{float32} & $3.28\times 10^{-3}$ & $1.58\times 10^{-3}$ & $1.63\times 10^{-5}$ \\
 & & \texttt{float64} & $3.57\times 10^{-3}$ & $1.92\times 10^{-3}$ & $1.29\times 10^{-5}$ \\
\bottomrule
\end{tabular}
\end{table}

\begin{figure}[htbp]
\centering
\includegraphics[width=1.0\linewidth]{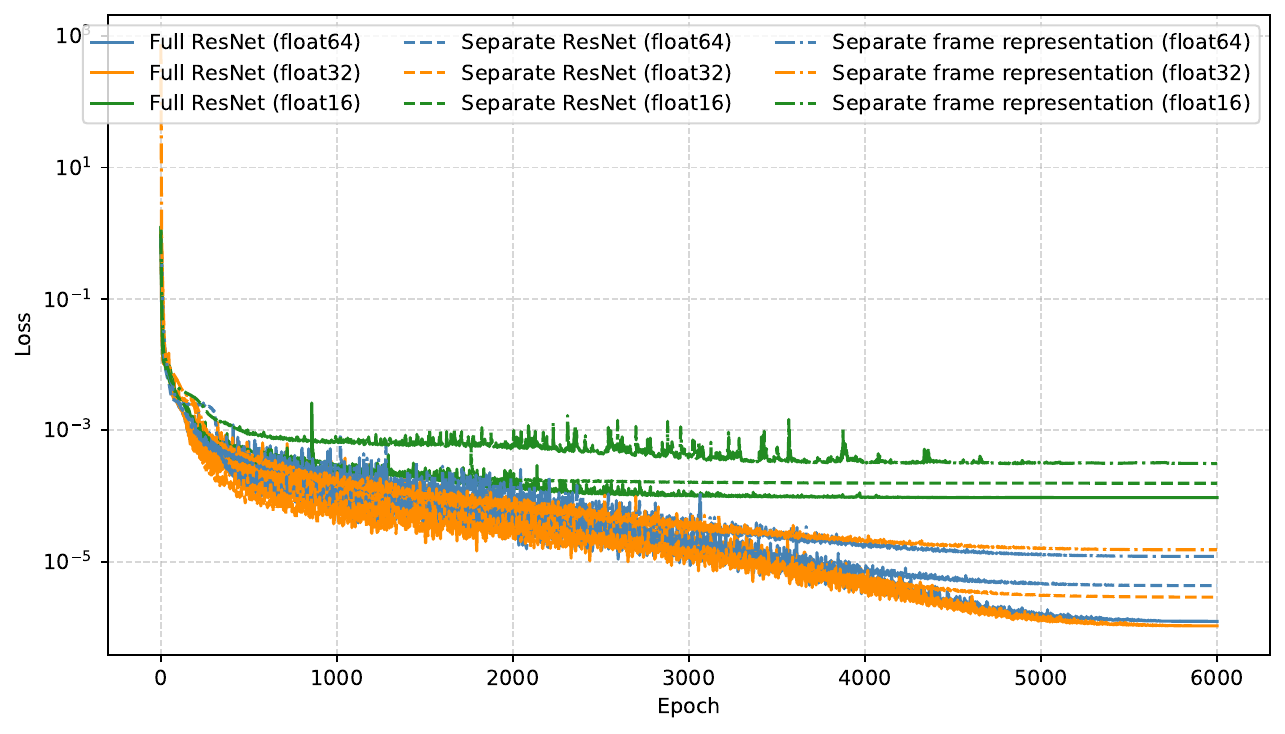}
\caption{Loss curves of numerically stable preconditioning with different network architectures under different computing precisions.}
\label{fig:fosls:stable:preconditioning_sep}
\end{figure}

\section{Conclusions}

In this work, we have investigated the effect of preconditioning on the training of neural networks for the approximation of parameter-dependent PDEs with focus on FOSLS formulations of elliptic problems. While compared to the formulation without preconditioning, we observe a striking improvement of training results for the Adam optimizer, this requires working with sufficient machine precision. We have thus proposed a modified representation of preconditioned matrices that is still easy to implement efficiently that yields numerical stability and can be used for computations using half-precision floating point numbers. 

Adam clearly outperforms SGD in all our tests with preconditioned problems. At the same time, methods such as L-BFGS or NGD do not yield any further improvement over Adam. The favorable optimization results obtained with Adam and preconditioning also show that the low-rank ResNet architecture adapted from \cite{BachmayrDahmenOster2024} that we use can achieve total solution errors on the order of the discretization error expected for the given mesh size. Variants of this architecture that use separate representations for coefficients on different frame levels did not show an improvement over a simpler joint ResNet approximation of all frame coefficients.

\section*{Acknowledgements}

The authors thank Polina Sachsenmaier and Konrad Westfeld for support with initial numerical tests.

Co-funded by the European Union (ERC, COCOA, 101170147). Views and opinions expressed are however those of the authors only and do not necessarily reflect those of the European Union or the European Research Council. Neither the European Union nor the granting authority can be held responsible for them.

\bibliographystyle{amsplain}
\bibliography{precond_stability}

\end{document}